\documentclass{article}

\PassOptionsToPackage{numbers, compress}{natbib}

\usepackage[final]{neurips_2024}

\usepackage[utf8]{inputenc} 
\usepackage[T1]{fontenc}    
\usepackage{hyperref}       
\usepackage{url}            
\usepackage{booktabs}       
\usepackage{amsfonts}       
\usepackage{nicefrac}       
\usepackage{microtype}      
\usepackage{xcolor}         

\usepackage{amsmath}
\usepackage{amssymb}
\usepackage{mathtools}
\usepackage{amsthm}

\usepackage{graphicx}
\usepackage{subfigure}

\usepackage[capitalize,noabbrev]{cleveref}

\theoremstyle{plain}
\newtheorem{theorem}{Theorem}[section]

\newtheorem{lemma}[theorem]{Lemma}
\newtheorem{corollary}[theorem]{Corollary}
\theoremstyle{definition}

\newtheorem{assumption}[theorem]{Assumption}
\theoremstyle{remark}
\newtheorem{remark}[theorem]{Remark}

\usepackage{algorithm}
\usepackage{algorithmicx}
\usepackage{algpseudocode}

\newcommand{\T}{\top}

\title{B-ary Tree Push-Pull Method is Provably Efficient \\
for Distributed Learning on Heterogeneous Data}

\author{%
  Runze You\\
  School of Data Science\\
  The Chinese University of Hong Kong, Shenzhen (CUHK-Shenzhen)\\
  \texttt{runzeyou@link.cuhk.edu.cn} \\
  \And
  Shi Pu \\
  School of Data Science \\
  The Chinese University of Hong Kong, Shenzhen (CUHK-Shenzhen)\\
  \texttt{pushi@cuhk.edu.cn} \\
}

\newcommand{\bA}{{\mathbf{A}}}

\newcommand{\bB}{{\mathbf{B}}}

\newcommand{\cC}{{\mathcal{C}}}

\newcommand{\cD}{{\mathcal{D}}}

\newcommand{\bfe}{{\mathbf{e}}}

\newcommand{\cE}{{\mathcal{E}}}

\newcommand{\bF}{{\mathbf{F}}}
\newcommand{\cF}{{\mathcal{F}}}

\newcommand{\bG}{{\mathbf{G}}}
\newcommand{\cG}{{\mathcal{G}}}

\newcommand{\bI}{{\mathbf{I}}}
\newcommand{\cI}{{\mathcal{I}}}

\newcommand{\cN}{{\mathcal{N}}}

\newcommand{\cO}{{\mathcal{O}}}

\newcommand{\bQ}{{\mathbf{Q}}}

\newcommand{\cR}{{\mathcal{R}}}

\newcommand{\cS}{{\mathcal{S}}}

\newcommand{\bu}{{\mathbf{u}}}
\newcommand{\bU}{{\mathbf{U}}}

\newcommand{\cU}{{\mathcal{U}}}

\newcommand{\bv}{{\mathbf{v}}}
\newcommand{\bV}{{\mathbf{V}}}

\newcommand{\bW}{{\mathbf{W}}}

\newcommand{\bx}{{\mathbf{x}}}

\newcommand{\bX}{{\mathbf{X}}}

\newcommand{\bY}{{\mathbf{Y}}}

\newcommand{\bZ}{{\mathbf{Z}}}

\newcommand{\bxi}{{\boldsymbol{\xi}}}

\newcommand{\bPi}{{\boldsymbol{\Pi}}}

\newcommand{\expect}{\mathbb{E}}
\newcommand{\reals}{\mathbb{R}}

\newcommand{\prt}[1]{\left(#1\right)}
\newcommand{\brk}[1]{\left[#1\right]}
\newcommand{\crk}[1]{\left\{#1\right\}}

\newcommand{\norm}[1]{\left\|#1\right\|}

\newcommand{\mone}{\mathbf{1}}

\begin{document}

\maketitle

\begin{abstract} 
    This paper considers the distributed learning problem where a group of agents cooperatively minimizes the summation of their local cost functions based on peer-to-peer communication.
    Particularly, we propose a highly efficient algorithm, termed ``B-ary Tree Push-Pull'' (BTPP), that employs two B-ary spanning trees for distributing the information related to the parameters and stochastic gradients across the network. The simple method is efficient in communication since each agent interacts with at most $(B+1)$ neighbors per iteration. More importantly, BTPP achieves linear speedup for smooth nonconvex and strongly convex objective functions with only $\tilde{O}(n)$ and $\tilde{O}(1)$ transient iterations, respectively,
    significantly outperforming the state-of-the-art results to the best of our knowledge. Our code is available at \href{https://github.com/ryou98/BTPP}{https://github.com/ryou98/BTPP}.
\end{abstract}

\section{Introduction}

In this paper, we consider a group of agents, labeled as $\cN := \{1, 2,\ldots, n\}$, in which each agent $i$ holds its own local cost function $f_i : \reals^p \rightarrow \reals$ and communicates only within its direct neighborhood. We investigate how the agents collaborate to locate $x\in \reals^p$ that minimizes the average of all the cost functions:
    \begin{equation}
        \label{eq:obj}
        \min_{x \in \reals^p } f(x)  \prt{=  \frac{1}{n} \sum_{i=1}^{n} f_i(x) },
    \end{equation}
    where $f_i(x) := \expect_{\xi_i\sim \cD_i}\brk{F_i(x;\xi_i)}$. Here $\xi_i$ denotes the local data of agent $i$ that follows the local distribution $\cD_i$. Data heterogeneity exists if $\crk{\cD_i}_{i=1}^{n}$ are not identical.

    To solve problem \eqref{eq:obj}, we assume each agent $i$ queries a stochastic oracle ($\cS\cO$) to obtain noisy gradient samples.
   Stochastic gradients appear in many areas including online distributed learning \cite{recht2011hogwild,dean2012large}, reinforcement learning \cite{mnih2013playing, lillicrap2015continuous}, generative modeling \cite{goodfellow2014generative, kingma2019introduction}, and parameter estimation \cite{bottou2012stochastic, srivastava2014dropout}. \cref{a.var} ensures that the gradient estimator $g_i(x;\xi_i)$ remains unbiased with a bounded variance for any given $\bx$, while independent samples $\xi_i$ are gathered continuously over time. In addition, the assumption is critical in simulation-based optimization as gradient estimation often encounters noise from multiple sources, such as modeling and discretization errors, or limitations due to finite sample sizes in Monte-Carlo methods \cite{kleijnen2018design}. 

    Modern optimization and machine learning typically involve tremendous data samples and model parameters. The scale of these problems calls for efficient distributed algorithms across multiple computing nodes. Recently, distributed algorithms dealing with problem \eqref{eq:obj} have been studied extensively in the literature; see, e.g., \cite{pu2021distributed,lian2017can,ding2023dsgd,ying2021exponential}.
    Traditional distributed learning approaches typically follow a centralized master-worker setup, where each worker node communicates with a (virtual) central server \cite{li2014scaling}.  
    However, such a communication pattern incurs significant communication overheads and long latency, especially when the training requires a large number of computing nodes. 
    
    Decentralized learning is an emerging paradigm to save communication costs, where the computing nodes are connected through a certain network topology (e.g., ring, grid, hypercube).
    Decentralized algorithms do not rely on central servers: the agents maintain the similarity among their copies of model parameters through peer-to-peer messages passing by communicating locally with immediate neighbors in the network. Such a setup allows each node to communicate with only a few peers and hence incurs much lower communication overhead \cite{assran2019stochastic}. Moreover, it offers strong promise for new applications, allowing agents to collaboratively train a model while respecting the data locality and privacy of each contributor.
    
    Specifically, in decentralized stochastic gradient methods, the agents share their local stochastic gradient updates through gossip communication \cite{xiao2004fast}. At every iteration, the local updates are sent to the neighbors of each agent who iteratively propagate the information through the network. Typically, the agents employ iterative gossip averaging of their neighbors' models with their own, where the averaging weights are chosen to ensure asymptotic uniform distribution of each update across the network. However, local averaging is less effective in ``mixing'' information which makes decentralized algorithms converge slower than their centralized counterparts. 
    Generally speaking, the network topology determines both the number of per-iteration communications and the convergence rates of decentralized algorithms, leading to a trade-off. For example, a densely-connected graph enables decentralized methods to converge faster but results in less efficient communication since each node needs to communicate with more neighbors. By contrast, a sparsely-connected topology results in a slower convergence rate but also reduces the per-iteration communication cost \cite{pu2021distributed,pu2021sharp,yuan2023removing}. In particular, for smooth and non-convex objective functions, it has been shown that decentralized stochastic gradient methods (with arbitrary topology) can achieve the same convergence rate as the centralized SGD method, but only after an initial period of iterations has passed \cite{lian2017can,ying2021exponential,pu2020asymptotic}. The number of transient iterations (transient time) heavily depends on the network topology, and thus a practical decentralized stochastic gradient algorithm should aim to minimize the transient time while keeping the number of per-iteration communications small (e.g., over a a sparsely-connected topology).
    Such an observation has motivated several recent works, which consider network topologies with $\Theta(1)$ per-iteration communications (or degree) for each node; see, e.g., \cite{ying2021exponential,song2022communication}.
    
    This work considers an alternative mechanism to gossip averaging, called ``B-ary Tree Push-Pull'' (BTPP), inherited from the Push-Pull method in \cite{pu2020push,xin2018linear}. Rather than relaying the messages over one graph at every iteration, BTPP uses two B-ary trees ($\cG_{\cR}$ and $\cG_{\cC}$) to spread the information about the parameters and the stochastic gradients, respectively. Each agent assigned in the B-ary tree acts as a worker on an assembly line. The model parameters are transmitted through the graph $\cG_{\cR}$ from the parent nodes to the child nodes. Meanwhile, the stochastic gradients are computed under the current model parameters and accumulated through the inverse graph of $\cG_{\cR}$ denoted as $\cG_{\cC}$. 
    BTPP can be viewed as a semi-(de)centralized approach given the critical role of node $1$.
    Notably, the corresponding mixing matrices of $\cG_{\cR}$ and $\cG_{\cC}$ only consist of $0$'s and $1$'s, which together with the B-ary Tree topology design, results in high algorithmic efficiency. We show BTPP achieves an $\tilde{\cO}(n)$ transient time under smooth nonconvex objective functions with $\Theta(1)$ per-iteration communications for each agent. By comparison, the state-of-the-art transient time is $\cO(n^3)$ (see \cref{tab:topo}). 

    \subsection{Related Works}

\paragraph{Decentralized Learning} Decentralized Stochastic Gradient Descent (DSGD) type algorithms are increasingly popular for accelerating the training of large-scale machine learning models \cite{lian2017can,ying2021exponential,koloskova2019decentralized} . These algorithms have been adapted under a range of practical settings, including those discussed in \cite{assran2019stochastic,lin2021quasi}. However, DSGD suffers from data heterogeneity \cite{koloskova2020unified}, which triggers more advanced techniques such as EXTRA \cite{shi2015extra}, Exact-Diffusion/$\text{D}^2$ \cite{li2019decentralized}, and gradient tracking \cite{pu2021distributed}. The Push-Pull method \cite{pu2020push,xin2018linear} which enjoys broad topological requirements was introduced for deterministic decentralized optimization under strongly convex objectives. This work particularly takes advantage of the flexibility in the network design of Push-Pull, utilizing the B-ary tree family, while considering stochastic gradients for minimizing smooth nonconvex objectives.

\paragraph{Topology Design} Decentralized stochastic gradient algorithms often rely on gossip averaging over various topologies such as rings, grids, and tori \cite{nedic2018network}. The hypercube graph \cite{trevisan2017lecture} strikes a balance between the communication efficiency and the consensus rate, but the network size is constrained to be the power of two.  
The work in \cite{ying2021exponential} re-examined the static exponential graph with $\Theta(\ln(n))$ degree and introduced a one-peer exponential graph with $\Theta(1)$ degree while preserving the consensus properties under the specific requirement of $n$. The paper \cite{takezawa2023beyond} proposed a base-($k+1$) graph as an enhancement that achieves similar convergence rate as in \cite{ying2021exponential} under arbitrary network size by sequentially employing multiple graph topologies (splitting an all-connected graph into $\Theta(\ln(n))$ different subgraphs). DSGD-CECA \cite{ding2023dsgd} requires roughly $\lceil\log_2(n)\rceil$  rounds of message passing for global averaging with $\Theta(n)$ network topologies. 
OD(OU)-EquiDyn \cite{song2022communication} introduces algorithms that employ various topologies to achieve network-size independent consensus rates. 
RelaySGD \cite{vogels2021relaysum} offers a relay-based algorithm that ensures $\Theta(1)$ per-iteration communication across different topologies. 

The above-mentioned methods all enjoy comparable convergence rates with centralized SGD (and thus achieves linear speedup) when the number of iterations $T$ is large enough. The transient times are generally in the order of $\tilde{\cO}(n^3)$  under smooth nonconvex objectives (see \cref{tab:topo}) and $\tilde{\cO}(n)$ under smooth strongly convex  objectives (see \cref{tab:convex}). 

Note that the above works and this paper generally consider training machine learning modes within high-performance data-center clusters, in which the network topology can be fully controlled: any two nodes can directly communicate over the network when necessary. By comparison, in some other scenarios, the underlying network topology is fixed, and the communication between two nodes is constrained (e.g., in wireless sensor networks, internet of vehicles, etc).

\begin{table*}
    \vskip 0.15in
    \begin{center}
    \begin{small}
        \begin{sc}
    \begin{tabular}{ccccc}
        \toprule
        Algorithm   & Per-iter Comm. & Size $n$ & Based Graph  & Trans. Iter.\\
        \midrule
        $D^2$ (Ring) \cite{tang2018d}  & $\Theta(1)$     & arbitrary & 1  & $\cO(n^{11})$  \\
        DSGD (Ring)  \cite{nedic2018network}           & $\Theta(1)$     & arbitrary & 1  & $\cO(n^7)$\\
        Hypercube  \cite{trevisan2017lecture}      & $\Theta(\ln(n))$& power of 2& 1 & $\tilde{\cO}(n^3)$\\
        Static Exp. \cite{ying2021exponential}     &  $\Theta(\ln(n))$& arbitrary & 1 & $\tilde{\cO}(n^3)$ \\
        O.-P. Exp.  \cite{ying2021exponential} & 1               & power of 2& $\Theta(\ln(n))$ & $\tilde{\cO}(n^3)$\\
        RelaySGD    \cite{vogels2021relaysum}    & $\Theta(1)$& arbitrary & 1 & $\cO(n^3)$\\
        OD(OU)-EquiDyn \cite{song2022communication} & 1 & arbitrary & $\Theta(n)$ & $\cO(n^3)$\\
        DSGD-CECA \cite{ding2023dsgd} &  $\Theta(1)$ & arbitrary & $\Theta(\ln(n))$ & $\tilde{\cO}(n^3)$\\
        Base-($k+1$) \cite{takezawa2023beyond} & $\Theta(1)$& arbitrary &$\Theta(\ln(n))$ & $\tilde{\cO}(n^3)$  \\
        \textbf{BTPP  (Ours)} &  $ \boldsymbol{ \Theta(1)}$   & \textbf{arbitrary} & \textbf{2}  & $\boldsymbol{\tilde{\cO}(n)}$ \\
        \bottomrule
      \end{tabular}
        \end{sc}
    \end{small}
        
    \end{center}
    \caption{\label{tab:topo} Comparison of different algorithms for distributed stochastic optimization under smooth nonconvex objectives. 
    ``Per-iter Comm.'' denotes the number of per-iteration communications or neighbors (degree) for each agent.
    ``Based Graph'' represents the number of required graph topologies during the entire training procedure. 
    ``Trans. Iter.'' represents the number of transient iterations. The notation $\tilde{\cO}(\cdot)$ hides all the polylogarithmic factors.
    }
    \vskip -0.1in
\end{table*}

\begin{table*}
    \vskip 0.15in
    \begin{center}
    \begin{small}
        \begin{sc}
    \begin{tabular}{ccccc}
        \toprule
        Algorithm   & Per-iter Comm. & Size $n$ & Based Graph  & Trans. Iter.\\
        \midrule
        DSGD (Ring)  \cite{nedic2018network}           & $\Theta(1)$     & arbitrary & 1  & $\tilde{\cO}(n^5)$\\
        Static Exp. \cite{ying2021exponential}     &  $\Theta(\ln(n))$& arbitrary & 1 & $\tilde{\cO}(n)$ \\
        O.-P. Exp.  \cite{ying2021exponential} & 1               & power of 2& $\Theta(\ln(n))$ & $\tilde{\cO}(n)$\\
        RelaySGD    \cite{vogels2021relaysum}    & $\Theta(1)$& arbitrary & 1 & $\tilde{\cO}(n)$\\
        OD(OU)-EquiDyn \cite{song2022communication} & 1 & arbitrary & $\Theta(n)$ & $\tilde{\cO}(n)$\\
        \textbf{BTPP  (Ours)} &  $ \boldsymbol{ \Theta(1)}$   & \textbf{arbitrary} & \textbf{2}  & $\boldsymbol{\tilde{\cO}(1)}$ \\
        \bottomrule
      \end{tabular}
        \end{sc}
    \end{small}
        
    \end{center}
    \caption{\label{tab:convex} Comparison of different algorithms for distributed stochastic optimization under smooth strongly convex objectives. 
    The notation $\tilde{\cO}(\cdot)$ hides all the polylogarithmic factors inheriting from \cite{song2022communication, koloskova2020unified}.
    }
    \vskip -0.1in
\end{table*}

\subsection{Main Contribution}

This paper introduces a novel distributed stochastic gradient algorithm, termed ``B-ary Tree Push-Pull'' (BTPP), which is provably efficient for solving the distributed learning problem \eqref{eq:obj} under arbitrary network size.  
The main contribution is summarized as follows:
\begin{itemize}
    \item BTPP incurs a $\Theta(1)$ communication overhead per-iteration for each agent. Specifically, any agent in the network communicates with at most $(B+1)$ neighbors, where $B$ can be freely chosen to fit different settings. Generally speaking, larger $B$ increases the per-iteration communication cost but reduces the transient time at the same time.
    \item We show BTPP enjoys an $\tilde{\cO}(n)$ transient time or iteration complexity under smooth nonconvex objectives and an $\tilde{\cO}(1)$ transient time or iteration complexity under smooth strongly convex objectives. Such results outperform the baselines: see \cref{tab:topo} and \cref{tab:convex}. The improvement is significant since the transient time greatly impacts the algorithmic performance, especially under large $n$.
    \item The convergence analysis for BTPP is non-trivial, partly due to the fact that the algorithm admits two different network topologies for communicating the model parameters and the (stochastic) gradient trackers respectively. Instead of constructing the induced matrix norms $\|\cdot\|_{\cR}$ and $\|\cdot\|_{\cC}$ as in \cite{pu2020push}, the analysis is performed under $\norm{\cdot}_2$ and $\norm{\cdot}_F$ only by carefully treating the matrix products and related terms. 
\end{itemize}

\subsection{Notation and Preliminaries}

Throughout the paper, vectors default to columns if not otherwise specified. Let each agent $i$ hold a local copy $x_i \in \reals^p$ of the decision variable and an auxiliary variable $y_i\in \reals^p$. Their values at iteration $k$ are denoted by $x_i^{(k)}$ and $y_i^{(k)}$, respectively. We let 
$\bX = \brk{x_1, x_2, \cdots,x_n}^{\top} \in \reals^{n\times p}$, $\bY= \brk{y_1, y_2, \cdots,y_n}^{\T} \in \reals^{n\times p}$,
and $\mone$ denotes the column vector with all entries equal to $1$. We also define the aggregated gradients at the local variables as 
$\nabla F(\bX) :=  \brk{\nabla f_1 (x_1),\nabla f_2 (x_2),  \cdots, \nabla f_n(x_n) }^{\T} \in \reals^{n\times p}$,
 where $F(\bX) := \sum_{i=1}^n f_i(x_i)$.
In addition, denote 
$\bxi := \brk{\xi_1, \xi_2, \cdots, \xi_n}^{\T}$, $\bG(\bX,\bxi):= \brk{g_1(x_1,\xi_1), g_2(x_2,\xi_2), \cdots, g_n(x_n,\xi_n)}^{\T} \in \reals^{n\times p}$.
For the conciseness of presentation, we also use $\bG^{(t)}$ to represent $\bG(\bX^{(t)},\bxi^{(t)})$. The term $\langle a, b\rangle$ stands for the inner product of two vectors $a,b\in \reals^p$. For matrices, $\norm{\cdot}_2$ and $\norm{\cdot}_F$ represent the spectral norm and the Frobenius norm respectively, which degenerate to the Euclidean norm for vectors. For simplicity, any square matrix with power $0$ is the unit matrix $\bI$ with the same dimension if not otherwise specified.

 We assume each agent $i$ is able to obtain noisy gradient samples of the form $g_i(x,\xi_i)$ that satisfies the following assumption.
    \begin{assumption}
        \label{a.var}
        For all $i \in \cN$ and $x\in\reals^p$, each random vector $\xi_i$ is independent and 
   $$
   \begin{aligned}
       & \expect_{\xi_i\sim\cD_i} \brk{ g_i(x,\xi_i)| x}  = \nabla f_i(x) ,\ \expect_{\xi_i \sim \cD_i} \brk{ \norm{g_i(x,\xi_i) - \nabla f_i(x)}^2 |x}  \le \sigma^2
   \end{aligned}
   $$
   for some $\sigma^2 > 0$.
    \end{assumption}

Regarding the individual objective functions $f_i$, we make the following standard assumption.
\begin{assumption}
    \label{a.smooth}
        Each $f_i(x) : \reals^p \rightarrow \reals$ is lower bounded with $L$-Lipschitz continuous gradients, i.e., for any $x, x' \in \reals^p$,
        $$
        \norm{ \nabla f_i(x) - \nabla f_i(x')} \le L\norm{x - x'}.
        $$
\end{assumption}
We also consider the following standard assumption regarding strongly convexity.
\begin{assumption}
    \label{a.convex}
    For any $x,y\in \reals^p$,
    $$
    f(y) \ge f(x) + \left\langle\nabla f(x), y-x \right\rangle + \frac{\mu}{2}\left\| 
y-x \right\|^2.
    $$
\end{assumption}
Denote $f^*:= \min_{x \in \reals^p } f(x)$. Let $x^* = \arg\min_x f(x)$ if \cref{a.convex} holds.

A directed graph $\cG(\cN, \cE)$ consists of a set of $n$ nodes $\cN$ and a set of directed edges $\cE\subseteq \cN\times \cN$, where an edge $(j,i)\in \cE$ indicates that node $j$ can directly send information to node $i$. To facilitate the local averaging procedure, each graph can be associated with a non-negative weight matrix $\bW = \brk{w_{ij}} \in \reals^{n\times n}$, whose element $w_{ij}$ is non-zero only if $(j, i)\in \cE$. 
Similarly, a non-negative weight matrix $\bW$ corresponds to a directed graph denoted by $\cG_\bW$.
For a given graph $\cG_\bW$, the in-neighborhood and out-neighborhood of node $i\in \cN$ are given by $\cN_{\bW,i}^{in}:= \{j\in\cN: (j,i)\in\cE\}$ and $\cN_{\bW,i}^{out}:= \{j\in\cN: (i,j)\in\cE\}$, respectively. The degree of node $i$ is the number of its in-neighbors or out-neighbors. For example, in a one-peer graph, the degree of each node is at most $1$.

\subsection{Organization of the Paper}
The rest of this paper is organized as follows. In \cref{sec:algorithm}, we introduce the B-ary Tree Push-Pull algorithm and present its main convergence results. The sketch of analysis is presented in \cref{sec:analysis}, and numerical experiments are provided in \cref{sec:numerical}. We conclude the paper in \cref{sec:conclusion}.

\section{B-ary Tree Push-Pull Method}
\label{sec:algorithm}

\subsection{Communication Graphs}

The proposed B-ary Tree Push-Pull method makes use of two spanning trees as communication graphs: $\cG_{\cR}$ and $\cG_{\cC}$, which correspond to two mixing matrices $\cR$ and $\cC$, respectively. Specifically, we consider B-ary tree graphs with arbitrary number of nodes $n$ and depth $d$. The root node is labeled as $1$ for convenience, and we index the nodes layer-by-layer. The additional nodes are placed at the last layer if the tree is not full. 
\cref{fig:graph2} illustrates the assignment of 10 nodes when $B=2$.
In the Pull Tree $\cG_{\cR}$ (the left ones), each node has $1$ parent node and $B$ child nodes (except the ones in the last layer). The root node $1$ has no parent node.
In the Push Tree $\cG_{\cC}$ (the right ones), each node has  $1$ child node and $B$ parent nodes (except the ones in the last layer). 
It can be seen that the tree $\cG_{\cC}$ is identical to $\cG_{\cR}$ with all the edges reversing directions. Note that only node 1 has a self-loop.

\begin{figure}[ht]
\vskip 0.2in
\begin{center}
    \includegraphics[width=0.75\textwidth]{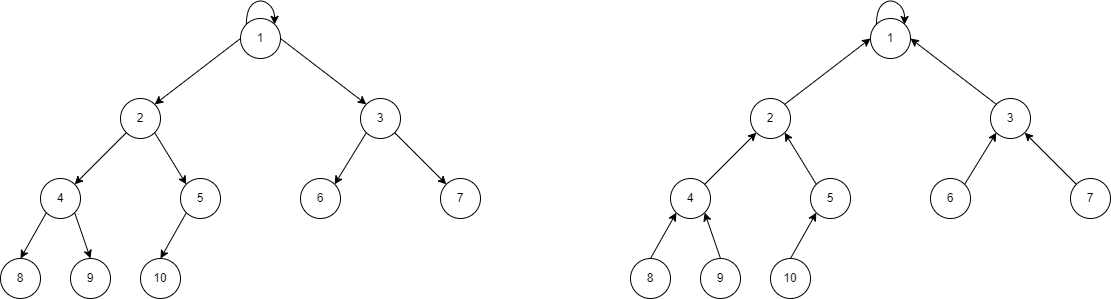}
    \caption{\label{fig:graph2}Two spanning trees with $10$ nodes when $B=2$. On the left is $\cG_{\cR}$, and the right one is $\cG_{\cC}$.}
    \label{fig:graph}
\end{center}
\vskip -0.2in
\end{figure}

\subsection{Algorithm}

We consider the following distributed stochastic gradient method (\cref{alg:pp}) for solving problem \eqref{eq:obj}. At every iteration $t$, each agent $i$ pulls the state information from its in-neighborhood $\cN_{\cR,i}^{in}$, pushes its (stochastic) gradient tracker $y_i$ to the out-neighborhood $\cN_{\cC,i}^{out}$, and updates its local variables $x_i$ and $y_i$ based on the received information. The agents aim to find the $\epsilon$-stationary point jointly by performing local computation and exchanging information through two spanning trees.

\begin{algorithm}[htbp]
    \caption{B-ary Tree Push-Pull Method (BTPP)}
    \label{alg:pp}
    \begin{algorithmic}[1]
        \State Each agent $i$ initializes with any arbitrary but identical $x_{i}^{(0)} = x^{(0)} \in \reals^p$, $y_i^{(0)} = g_i(x_i^{(0)}, \xi_i^{(0)}) \in \reals^p$ after drawing a random sample $\xi_i^{(0)}$, stepsize $\gamma>0$ and number of iterations $T$.
        \For{iteration $t = 0, 1, 2,\ldots, T-1$}
        \For{agent $i$ \textbf{in parallel}}
        \State Pull $x_j^{(t)} - \gamma y_{j}^{(t)}$ from each $j \in \cN_{\cR,i}^{in}$
        \State Push $y_i^{(t)}$ to each $j\in \cN_{\cC,i}^{out}$
        \State Independently draw a random sample $\xi_i^{(t+1)}$ 
        \State Update parameters through
        $$
        \begin{aligned}
            x_i^{(t+1)} & = \sum_{j \in \cN_{\cR,i}^{in}} \prt{x_j^{(t)} - \gamma y_{j}^{(t)}} \\
            y_i^{(t+1)} & = \sum_{j\in\cN_{\cC,i}^{in}} y_j^{(t)} + g_i(x^{(t+1)}_i;\xi_i^{(t+1)}) -  g_i(x^{(t)}_i;\xi_i^{(t)}) 
        \end{aligned}
        $$
        \EndFor
        \EndFor
        \State Output $x_{1}^{(T)}$.
    \end{algorithmic}
\end{algorithm}

 More specifically, in the pull tree $\cG_{\cR}$, each node $i$ pulls the updated model from its parent node along the tree. 
 Note that $\cN_{\cR,i}^{in}$ consists of only one node, the parent node.
 The Push Tree $\cG_{\cC}$ is the inverse of the Pull Tree, in which each node collects and aggregates the gradient trackers from its parent nodes. Due to the tree structure, only $y_1^{t}$ aggregates and tracks the stochastic gradients across the entire network, which will be made clear from the analysis.
 The implementation of the algorithm is rather simple. Taking node $2$ in \cref{fig:graph} as an example, we have $x_2^{(t+1)} = x_1^{(t)}-\gamma y_{1}^{(t)}$ and $y_2^{(t+1)} = y_4^{(t)} + y_5^{(t)} + g_2(x^{(t+1)}_2;\xi_2^{(t+1)}) - g_2(x^{(t)}_2;\xi_2^{(t)})$.
 
 We can write \cref{alg:pp} in the following compact form:
\begin{equation}
    \label{eq:ppSGD}
    \begin{aligned}
        \bX^{(t+1)} \ & = \cR \prt{\bX^{(t)} - \gamma \bY^{(t)}} \\
        \bY^{(t+1)} \ & = \cC \bY^{(t)} + \bG(\bX^{(t+1)},\bxi^{(t+1)}) - \bG(\bX^{(t)},\bxi^{(t)})
    \end{aligned}
\end{equation}
where $\bY^{(0)}=\bG(\bX^{(0)},\bxi^{(0)})$, and $\cR,\cC\in \reals^{n\times n}$ are non-negative matrices whose elements are given by
$$
\brk{\cR}_{i,j} = \left\{
    \begin{aligned}
        1 & \quad \text{if } i\in \left\{Bj+1 -B + [B] \right\}\cap [n]  \text{ or } i=j=1\\
        0 & \quad \text{otherwise}
    \end{aligned}
\right. 
$$
and $\cC = \cR^{\T}$ which corresponds to $\cG_{\cC}$, the inverse tree of $\cG_{\cR}$. It can be seen that $\cR$ is a row-stochastic matrix that only consists of $0$'s and $1$'s, and $\cC$ is column stochastic. For example, the mixing matrices corresponding to the graphs in \cref{fig:graph2} are given by

{\small$$
\cR = \prt{ \begin{matrix}
    1 &   &   &   &   &   &   &   &   &  \\
    1 &   &   &   &   &   &   &   &   &  \\
    1 &   &   &   &   &   &   &   &   &  \\
      & 1 &   &   &   &   &   &   &   &   \\
      & 1 &   &   &   &   &   &   &   &  \\
      &   & 1 &   &   &   &   &   &   &  \\
      &   & 1 &   &   &   &   &   &   &  \\
      &   &   & 1 &   &   &   &   &   &  \\
      &   &   & 1 &   &   &   &   &   &  \\
      &   &   &   & 1 &   &   &   &   &  \\
\end{matrix} }, \quad 
\cC = \prt{ \begin{matrix}
    1 & 1 & 1 &   &   &   &   &   &   &  \\
      &   &   & 1 & 1 &   &   &   &   &  \\
      &   &   &   &   & 1 & 1 &   &   &  \\
      &   &   &   &   &   &   & 1 & 1 &  \\
      &   &   &   &   &   &   &   &   &1 \\
      &   &   &   &   &   &   &   &   &  \\
      &   &   &   &   &   &   &   &   &  \\
      &   &   &   &   &   &   &   &   &  \\
      &   &   &   &   &   &   &   &   &  \\
      &   &   &   &   &   &   &   &   &  \\
\end{matrix} },
$$\normalsize}
where the unspecified elements are zeros.

\subsection{Main Results}
The main convergence properties of BTPP are summarized in the following two theorems, where the second result assumes strongly convexity on $f$.
\begin{theorem}
    \label{thm:convergence}
    For the BTPP algorithm outlined in \cref{alg:pp} implemented on B-ary tree graphs $\cG_{\cR}$ and $\cG_{\cC}$, assume \cref{a.var} and \cref{a.smooth} hold. Let $\gamma = \min\{\prt{\frac{\Delta_f}{3\sigma^2 L n (T+1)}}^{\frac{1}{2}}, \prt{\frac{\Delta_f}{1500n^2 d^6 \sigma^2 L^2 (T+1)}}^{\frac{1}{3}}, \frac{1}{100nd^3L}\}$. The following convergence result holds:
    \begin{equation}
        \label{eq:convergence}
        \begin{aligned}
            & \frac{1}{T+1} \sum_{t=0}^{T}\expect\norm{\nabla f(x_1^{(t)})}_2^2 
            \le \frac{32 \sqrt{\Delta_f \sigma^2 L}}{\sqrt{n(T+1)}} + \frac{240d^2 \prt{\sigma^2 L^2 \Delta_f^2}^{\frac{1}{3}}}{\prt{\sqrt{n} (T+1)}^{\frac{2}{3}}   } + \frac{800d^3 L \Delta_f }{T+1}  + \frac{\norm{\nabla \bF(\bX^{(0)})}_F^2}{n(T+1)},
        \end{aligned}
    \end{equation}
    where $\Delta_f := f(\bx_1^{(0)}) - f^*$ and $d = \lfloor\log_B(n)\rfloor$ represents the diameter of the graphs. 
\end{theorem}

\begin{remark}
    Based on the convergence rate in \eqref{eq:convergence} of BTPP, we can derive that when $T = \Theta(n\log^{12}(n))$, the term $\cO(\frac{1}{\sqrt{nT}})$ dominates the remaining terms up to a constant scalar. This implies that BTPP achieves linear speedup after $\cO\prt{n\log^{12}(n)}$ transient iterations.
\end{remark}

\begin{remark}
    \label{rmk:branch}
    The convergence rate in \eqref{eq:convergence} is related to the branch size $B$. For larger $B$, the diameter $d = \lfloor \log_B(n)\rfloor$ becomes smaller, which results in more efficient transmission of information and fewer transient iterations. However, the per-iteration communication cost is relatively larger. When $B$ is smaller, the communication burden for each agent at every iteration is lighter, but the transient time is larger. Therefore, in practice, the communication cost and convergence rate can be balanced by considering a proper $B$.
\end{remark}
\begin{theorem}
    \label{thm:convergence-convex}
  For the BTPP algorithm outlined in \cref{alg:pp} implemented on B-ary tree graphs $\cG_{\cR}$ and $\cG_{\cC}$, assume \cref{a.var}, \cref{a.smooth} and \cref{a.convex} hold. Let $\gamma = \min\crk{\frac{1}{100nd^2\kappa L}, \frac{16\log(n\prt{T+1}^2)}{n\prt{T+1}\mu}}$ and $T\ge 2d$. The following convergence result holds:
    \begin{equation}
    \label{eq:convergence-convex}
        \begin{aligned}
    &\mathbb{E}\left\|x_1^{(T)} - x^* \right\|^2 \le  \frac{2240\sigma^2 \log(n(T+1)^2)}{n\prt{T+1}\mu^2} + \frac{26880000 d^6 \kappa^2 \sigma^2\prt{\log(n(T+1)^2)}^2}{n\prt{T+1}^2\mu^2} \\
    & \quad + \max\crk{\exp(-\frac{T}{800d^2\kappa^2}), \frac{40}{n\prt{T+1}^2}}\prt{\left\|x_1^{(0)} - x^* \right\|^2 + \frac{1}{nL^2}\norm{\nabla\bF(\bX^{(0)})}_F^2  }.
\end{aligned}
    \end{equation}
\end{theorem}
\begin{remark}
    The convergence rate in \eqref{eq:convergence-convex} implies that
$
      \mathbb{E}\left\|x_1^{(T)} - x^* \right\|^2 \le \tilde{\mathcal{O}} \left( \frac{1}{nT} + \frac{1}{nT^2} + \exp\left( - T\right) \right),
$
where $\tilde{\mathcal{O}}$ hides the constants and polylogarithmic factors. The transient time is thus $\tilde{O}(1)$, i.e., the number of iterations before the term $\cO(\frac{1}{nT})$ dominates the remaining terms. Such a transient time also outperforms the state-of-the-art results.
\end{remark}

 \section{Analysis of B-ary Tree Push-Pull}
\label{sec:analysis}
In this section, we study the convergence of BTPP and prove \cref{thm:convergence} by analyzing the properties of the weight matrices $\cR$ and $\cC$, the evolution of the aggregated consensus error $\sum_{t=0}^{T}\expect \norm{\bPi_{\bu}\bX^{(t)}}_F^2$, and the expected inner products of the stochastic gradients between different layers. The approach is different from those employed in \cite{pu2020push,pu2021distributed,song2022communication}, where the analysis considers two special matrix norms related to $\mathcal{R}$ and $\mathcal{C}$, respectively. Such a distinction is because BTPP works with two B-ary trees and iterates in a layer-wise manner, while most other works consider connected graphs.

Our analysis starts with characterizing the weight matrices $\cR$ and $\cC$, as delineated in the following lemmas. It is important to note that for any given $n$ and a specific integer $B$, we can determine an integer $d$  satisfying $\frac{B^d - 1}{B-1}< n \le \frac{B^{d+1} - 1}{B-1}$ which is the diameter of the graphs. 

    Notice that $\mathcal{R}$ has a unique non-negative left eigenvector $\bu^{\T}$ (w.r.t. eigenvalue 1) with $\bu^{\T}\mone = n$. More specifically, $\bu = \brk{n,0,\cdots, 0}^{\T}$, which is also the unique right eigenvector of $\cC$ (w.r.t. eigenvalue 1), denoted by $\bv$ for the clarity of presentation.
Following the above observations, it is revealed in \cref{l.R0knorm} that the $2$-norm of the matrix $\cR - \frac{1}{n}\mone\bu^{\T}$ with exponent $k$ remains bounded by $\sqrt{n}$ and equals zero when $k$ exceeds $d-1$.
\begin{lemma}
    \label{l.R0knorm}
    Given a positive integer $k$, the 2-norm of the matrix $\cR^k - \frac{1}{n}\mone \bu^{\T}$ satisfies
    $$
    \|\cR^k- \frac{1}{n}\mone \bu^{\T}\|_2  \left\{\begin{aligned}
        \le \sqrt{n}& \quad k \le d-1 \\
         = 0 & \quad k\ge d
    \end{aligned}\right.
    $$
\end{lemma}
Similar result applies to the matrix $\cC^k - \frac{1}{n}\bv\mone^{\T}$. Consequently, we introduce the mixing matrices $\bPi_{\bu},\bPi_{\bv}$ based on the eigenvectors $\bu,\bv$, which play a crucial role in the follow-up analysis.
$$
\bPi_{\bu} := \bI - \frac{1}{n}\mone \bu^{\T}, \ \bPi_{\bv}:=\bI - \frac{1}{n}\bv\mone^{\T}.
$$
The following lemmas delineate the critical elements for constraining the average expected norms of the objective function as formulated in \eqref{eq:obj}, i.e., $\frac{1}{T+1} \sum_{t=0}^{T}\expect\norm{\nabla f(x_1^{(t)})}_2^2$.  \cref{l.X_diff} and \cref{l.PiX} provide bounds on the expressions
$\sum_{t=0}^{T}\expect\|\bar{\bX}^{(t+1)} - \Bar{\bX}^{(t)}\|_F^2$ and $\sum_{t=0}^{T} \norm{\bPi_{\bu}\bX^{(t)}}_F^2$, where $\bar{\bX}^{(t)} := \frac{1}{n}\mone\bu^\T \bX^{(t)}$.
\begin{lemma}
    \label{l.X_diff}
    Suppose \cref{a.var} holds and $\gamma\le \frac{1}{10ndL}$, we have the following inequality:
    $$
    \begin{aligned}
        & \sum_{t=0}^{T}\expect\|\bar{\bX}^{(t+1)} - \Bar{\bX}^{(t)}\|_F^2 \le  6\gamma^2 n^2\sigma^2(T+1) + 50\gamma^2n^2d^2L^2 \sum_{t=0}^{T}\expect \norm{\bPi_{\bu}\bX^{(t)}}_F^2 \\
        &\qquad + 6\gamma^2 n^2 d \norm{\nabla\bF(\bX^{(0)})}_F^2  + 15\gamma^2 n^3\sum_{t=0}^{T}\expect\norm{\nabla f(x_1^{(t)})}_2^2.
    \end{aligned}
    $$
\end{lemma}

\begin{lemma}
    \label{l.PiX}
    Suppose \cref{a.var} holds and $\gamma\le \frac{1}{40nd^2 L}$, we have for $d \ge 2$ that
    $$
    \begin{aligned}
        & \sum_{t=0}^{T} \expect\norm{\bPi_{\bu}\bX^{(t)}}_F^2 
        \le  300 \gamma^2 n^2 d^4 (T+1) \sigma^2 + 20\gamma^2 n^3 d^2 \sum_{t=0}^{T}\expect  \| \nabla f(x_1^{(t)})\|_2^2 \\
        & \qquad \qquad + 6nd\norm{\bPi_{\bu}\bX^{(0)}}_F^2 + 40\gamma^2 n^2 d^3 \norm{\nabla \bF(\bX^{(0)})}_F^2.
    \end{aligned}
    $$
\end{lemma}
From the design of BTPP, there is an inherent delay in the transmission of information from layer $k$ to layer $1$. As information traverses through the B-ary trees, the delay becomes evident. Specifically, for nodes at layer $k$, their information requires an additional $k$ iterations to successfully reach and impact node $1$, as demonstrated in \cref{l.inner}.
\begin{lemma}
    \label{l.inner}
    For any integer $t>1$, we have 
$$
\begin{aligned}
    & \sum_{m=1}^{\min\{t,d\}} \expect \left\langle \nabla f(x_1^{(t)}), \prt{\frac{\bu^{\T}}{n}-\mone^{\T}} \bA_m \prt{\bG^{(t-m)} - \nabla \bF(\bX^{(t-m)})}\right\rangle = 0,
\end{aligned}
$$
where $\bA_m = \cC^{m} - \cC^{m-1}$ and $\bA_{1} = \cC - \bI$.
\end{lemma}

Building on the preceding lemmas, we are in a position to establish the main convergence result for the BTPP algorithm. This involves upper bounding the expected norm for the gradient of the objective function evaluated at $x_1^{(t)}$. To show the result, we integrate the findings from \cref{l.X_diff}, \cref{l.PiX}, and \cref{l.inner}, as detailed in \cref{l.fbarX}.
\begin{lemma}
    \label{l.fbarX}
    Suppose \cref{a.var} and \cref{a.smooth} hold and $\gamma\le \frac{1}{100nd^3 L}$, we have
    $$
    \begin{aligned}
        &\frac{1}{T+1} \sum_{t=0}^{T}\expect\norm{\nabla f(x_1^{(t)})}_2^2 \le \frac{8\Delta_f}{\gamma n (T+1)} + 24\gamma  \sigma^2 L+ 20000\gamma^2 n d^6 \sigma^2 L^2 \\
        &\qquad + \frac{400 d^3 L^2\norm{\bPi_{\bu}\bX^{(0)}}_F^2 }{T+1} + \frac{56 \gamma d^3 L\norm{\nabla \bF(\bX^{(0)})}_F^2}{T+1}.
    \end{aligned}
    $$
\end{lemma}
\begin{remark}
    \cref{l.fbarX} implies that the transient time of BTPP under \cref{a.smooth} is influenced by the fourth term in the upper bound: $\frac{400d^3 L^2 \norm{\bPi_{\bu}\bX^{(0)}}_F^2}{T+1}$ which is related to the initial consensus error. Therefore, we initialize all the agents with the same solution $x^{(0)}$.
\end{remark}

 Under strong convexity of $f$, we have the following key lemma.
 \begin{lemma}
    \label{l.x-star}
   Suppose \cref{a.var}, \cref{a.smooth} and \cref{a.convex} hold, and $\gamma \le \frac{1}{100n d^2 \kappa L}$, we have
    $$
    \begin{aligned}
    &\mathbb{E}\left\|x_1^{(T)} - x^* \right\|^2 \le \left(1 - \frac{n\gamma \mu}{4} \right)^{T}\left\|x_1^{(0)} - x^* \right\|^2 \\
    & \quad + 7\gamma^2 n \sigma^2 \prt{T+1} + 21000 \gamma^3 n^2  d^6 \kappa L\sigma^2\prt{T+1}\\
    &\quad + 80\gamma^2 n d^3\prt{1-\frac{n\gamma\mu}{4}}^{T-d}\norm{\nabla \bF(\bX^{(0)})} + 420\gamma^3 n^2 d^3 \kappa L \norm{\bPi_{\bu}\bX^{(0)}}_F^2 ,
    \end{aligned}
    $$
    where $\kappa := L/\mu$ is the conditional number.
\end{lemma} 

\section{Numerical Results}
\label{sec:numerical}

This section presents experimental results to compare the B-ary Tree Push-Pull method with other popular algorithms on logistic regression with synthetic data and deep learning tasks with real data.

\subsection{Logistic Regression}

We compare the performance of BTPP against other algorithms listed in \cref{tab:topo} for logistic regression with non-convex regularization \cite{song2022communication}. The objective functions $f_i: \mathbb{R}^p\rightarrow \mathbb{R}$ are given by
$$
f_i(x) := \frac{1}{J} \sum_{j = 1}^{J} \ln\prt{1+ \exp(-y_{i,j}h_{i,j}^\T x)} + R\sum_{k = 1}^{p} \frac{x_{[k]}^2}{1+ x_{[k]}^2},
$$
where $x_{[k]}$ is the $k$-th element of $x$, and $\crk{(h_{i,j}, y_{i,j})}_{j=1}^J$ represent the local data kept by node $i$. 
To control the data heterogeneity across the nodes, we first let each node $i$ be associated with a local logistic regression model with parameter $\tilde{x}_i$ generated by $\tilde{x}_i = \tilde{x} + v_i$, where $\tilde{x} \sim \cN(0,\bI_p)$ is a common random vector, and $v_i\sim\cN(0, \sigma_h^2\bI_p)$ are random vectors generated independently. Therefore, $\{v_i\}$ decide the dissimilarities between $\tilde{x}_i$, and larger $\sigma_h$ generally amplifies the difference. After fixing $\crk{\tilde{x}_i}$, local data samples are generated that follow distinct distributions. For node $i$, the feature vectors are generated as $h_{i,j} \sim \cN(0, \bI_p)$, and $z_{i,j}\sim\cU(0,1)$. Then, the labels $y_{i,j}\in \crk{-1,1}$ are set to satisfy $z_{i,j}\le 1 + \exp(-y_{i,j}h_{i,j}^\T \tilde{x}_i)$. 

In the simulations, the parameters are set as follows: $n=100$, $p=500$, $J = 1000$, $R = 0.01$, and $\sigma_h = 0.8$. All the algorithms initialize with the same stepsize $\gamma = 0.3$, except BTPP, which employs a modified stepsize $\gamma/n$. Such an adjustment is due to BTPP's update mechanism, which incorporates a tracking estimator that effectively accumulates $n$ times the averaged stochastic gradients as the number of iterations increases. This can also be seen from the stepsize choice in Theorem \ref{thm:convergence}.\footnote{Note that this particular configuration results in slower convergence for BTPP during the initial $\mathcal{O}(d)$ iterations, which can be improved by using larger initial stepsizes.} Additionally, we implement a stepsize decay of 60\% every $100$ iterations to facilitate convergence. 

In Figure \ref{fig:LR}, the gradient norm is used as a metric to gauge the algorithmic performance of each algorithm.  The left panel of \cref{fig:LR} illustrates the comparative performance of various algorithms, highlighting that BTPP (in red) achieves faster convergence than the other algorithms with $\Theta(1)$ degree and closely approximates the performance of the centralized SGD algorithm (i.e., DSGD-FullyConnected). The right panel of \cref{fig:LR} demonstrates the behavior of BTPP when increasing the branch size $B$. It is observed that with larger $B$, the convergence trajectory of BTPP more closely aligns with that of centralized SGD, corroborating the prediction of the theoretical analysis.

\begin{figure}[ht]
    \centering
    \begin{minipage}{0.49\textwidth}
        \centering
        \includegraphics[width=\linewidth]{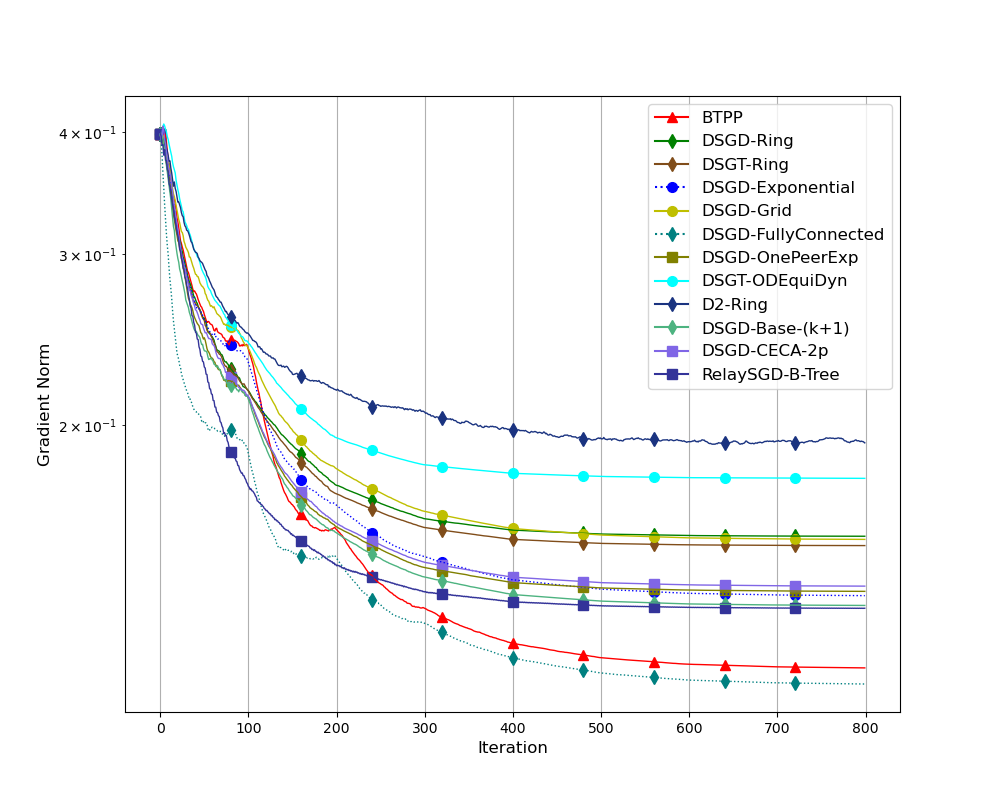} 
    \end{minipage}
    \hfill 
    \begin{minipage}{0.49\textwidth}
        \centering
        \includegraphics[width=\linewidth]{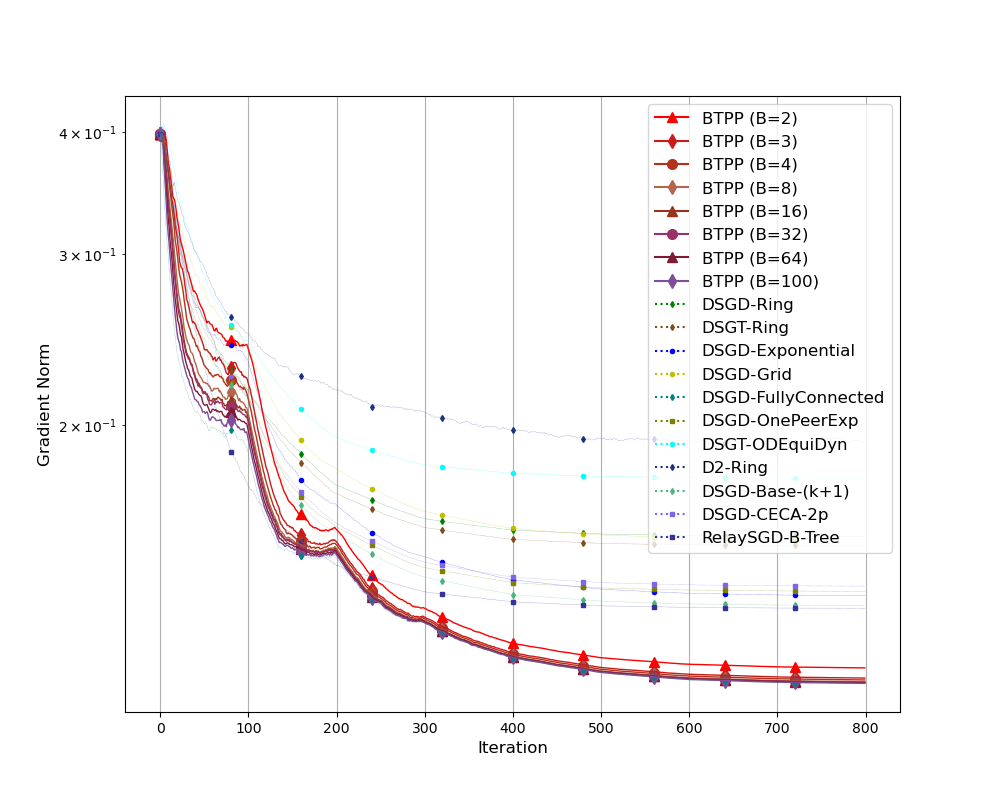} 
    \end{minipage}
    \caption{Left: performance of algorithms for logistic regression with nonconvex regularization, where the dotted lines correspond to algorithms whose degrees are not $\Theta(1)$. We let the branch size $B=2$ in BTPP, $\eta = 0.5$ in OD-EquiDyn, $k=2$ in Base-($k+1$), and perform RelaySGD on a binary tree graph for fairness. Right: performance of BTPP with different branch size $B$.}
    \label{fig:LR}
\end{figure}

\subsection{Deep Learning}
\label{sec:deep}
We apply BTPP and the other algorithms to solve the image classification task with CNN over \textbf{MNIST} dataset \cite{lecun2010mnist}. We run all experiments on a server with eight Nvidia RTX 3090 GPUs. The network contains two convolutional layers with max pooling and ReLu and two feed-forward layers. In particular, we consider a heterogeneous data setting, where data samples are sorted based on their labels and partitioned among the agents. 
The local batch size is set to $8$ with $24$ agents in total. The learning rate is $0.01$ for all the algorithms except BTPP (which employs a modified stepsize $\gamma/n$) for fairness. Additionally, the starting model is enhanced by pre-training using the SGD optimizer on the entire MNIST dataset for several iterations. 
Figure \ref{fig:mnist} illustrates the training loss and the test accuracy curves.
Comparing the performance of different algorithms, it can be seen that BTPP (in red)  and DSGT with ODEquiDyn (based on $\Theta(n)$ graphs) achieve faster convergence than the other algorithms with $\Theta(1)$ degree and closely approximate the performance of centralized SGD(i.e., DSGD-FullyConnected). 
\begin{figure}[ht]
    \centering
    \begin{minipage}{0.49\textwidth}
        \centering
        \includegraphics[width=\linewidth]{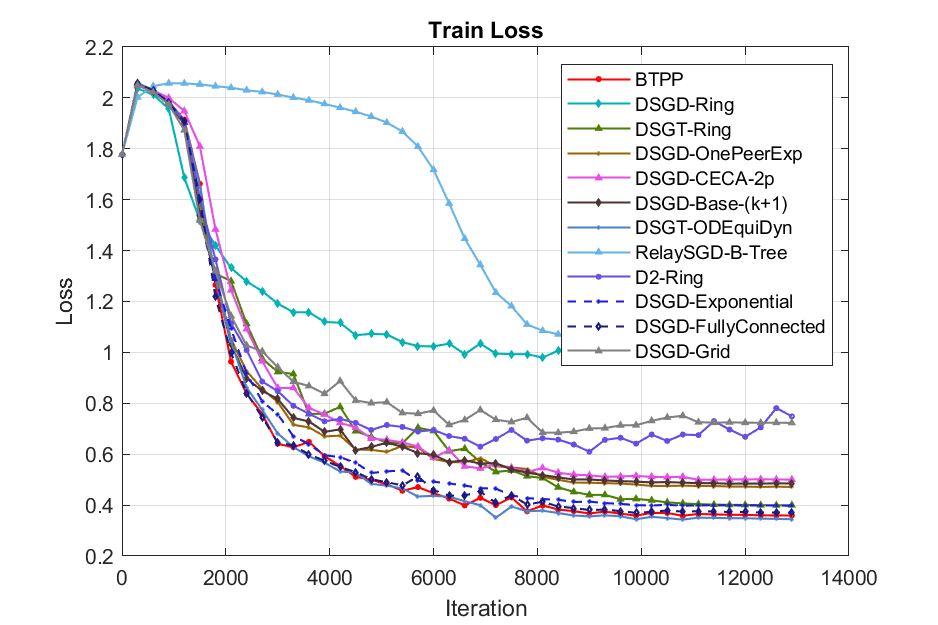} 
    \end{minipage}
    \hfill 
    \begin{minipage}{0.49\textwidth}
        \centering
        \includegraphics[width=\linewidth]{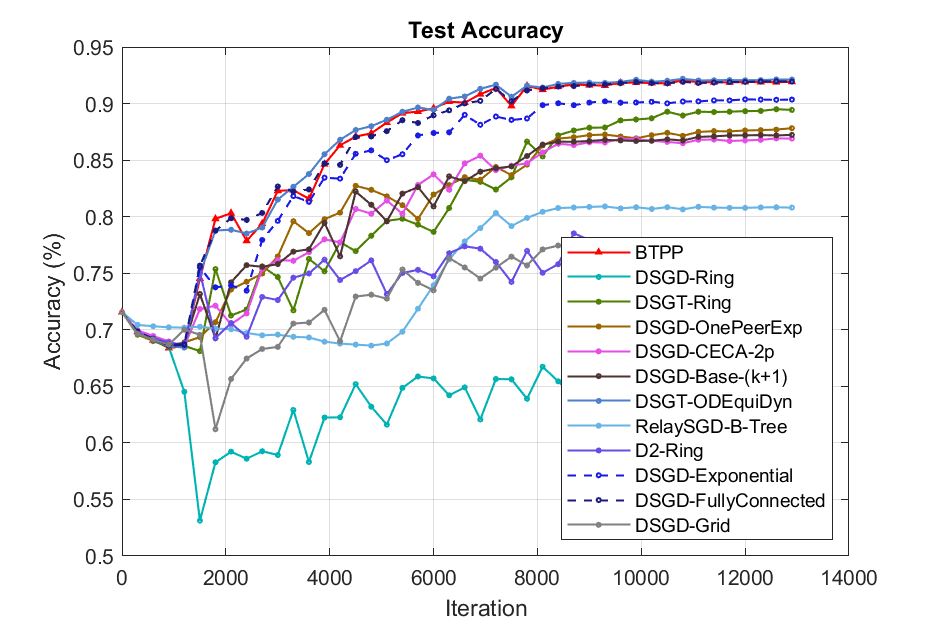} 
    \end{minipage}
    \caption{Train loss and test accuracy of different algorithms for training CNN on MNIST, where the dotted lines correspond to the algorithms whose degrees are not $\Theta(1)$. We perform BTPP with $B=2$, ODEquiDyn with $\eta = 0.5$, Base-($k+1$) with $k=2$, and RelaySGD on a binary tree graph for fairness.}
    \label{fig:mnist}
\end{figure}

\begin{remark}
 Higher accuracy can be achieved for BTPP and other methods when using the momentum technique, or when the data heterogeneity is removed, meaning that samples are randomly assigned to each agent. Additional experiments demonstrating the performance of various algorithms across different tasks and scenarios are provided in \cref{appendix:exp}.
\end{remark}
\section{Conclusions}
\label{sec:conclusion}

This paper proposes a novel algorithm for distributed learning over heterogeneous data, named BTPP. The convergence is theoretically analyzed for smooth non-convex stochastic optimization. The results demonstrate that, at the minimal communication cost per iteration, BTPP achieves linear speedup in the number of nodes $n$, and the transient times behaves as $\tilde{O}(n)$ and $\tilde{O}(1)$ respectively for smooth nonconvex and strongly convex objectives, outperforming the state-of-the-art results. Numerical experiments further validate the efficiency of BTPP.

\newpage
\bibliography{ref}
\bibliographystyle{siamplain}

\newpage
\appendix
\onecolumn

\section{Convergence Analysis of BTPP}
    In this section, we aim to demonstrate the convergence results of BTPP through a three-step process. First, we explore the key properties of matrices $\cR$ and $\cC$, acquainting readers with several operations that will be frequently utilized in the subsequent parts. Then, we introduce various technical tools essential for the analysis. Finally, we delve into proving the convergence results supported by a series of pertinent lemmas.

\subsection{Properties of the Weight Matrices}
\label{prop.matrix}

In this part, we first demonstrate that $\cR\in \reals^{n\times n}$ possesses a set of properties ( the matrix $\cC=\cR^{\T}$ shares similar properties). Then, we utilize the established tools to prove the crucial result presented in \cref{l.R0knorm}. Lastly, we provide clarifications on certain matrix operations that will be frequently employed in deriving the convergence results. 

It is important to note that for any given $n$ and specific integer $B$, the diameter of the corresponding B-ary tree graph $d$ (the distance from the last layer node to node 1) satisfies $\frac{B^{d} - 1}{B-1}< n \le \frac{B^{d+1} - 1}{B-1}$. 
To investigate the properties of $\cR$ and $\cC$, we will introduce the column vector $\bfe_{\cI} \in \reals^n$, where each element of $\bfe_{\cI}$ is equal to 1 for indices $i\in \cI$ and 0 otherwise. Define the index sets
$$
\begin{aligned}
    \cI_{1,k} & = \crk{ 1 : \frac{B^{k+1}-1}{B-1}}, \\
    \cI_{i,k} & = \crk{\prt{\frac{B^{k+1}-1}{B-1}  + (i-2)B^k + 1}: \prt{\frac{B^{k+1}-1}{B-1}  + (i-1)B^k}},
\end{aligned}
$$
where $k_1:k_2$ is the arithmetic progression from $k_1$ to $k_2$ with difference 1. We can then define the matrix $\bZ_k \in \mathbb{R}^{n\times n}$ as a composite of several column vectors arranged in the following format:
$$
\bZ_k =  \left[ \bfe_{\cI_{1,k} }, \bfe_{\cI_{2,k}}, \cdots,  \bfe_{\cI_{n,k}}\right].
$$
This closed-form expression of $\cR$ with any power $k$ is shown in \cref{l.R} which aids in developing further properties. 
\begin{lemma}
    \label{l.R}
    For the pull matrix $\mathcal{R}$ corresponding to the B-ary tree $\cG_{\cR}$, given any positive index $k$, we have
    $$
    \cR^k = \bZ_k.
    $$
\end{lemma}
\begin{proof}
    We prove the lemma by induction. First, it is obvious that $\cR = \bZ_1$ by the definition of $\cR$:
    $$
    \begin{aligned}
        & \cR_{ij} = 1 \\
        \text{ iff } & i \in \{Bj+1-B + [B]\}\cap [n] \text{ or } i=j=1 \\
        \text{ iff } & B(j-1) + 2 \le i \le Bj+1, \ i\in [n] \text{ or } i=j=1\\
        \text{ iff } & [Z_1]_{ij} = 1.
    \end{aligned}
    $$
    Now assume the statement is true for $k = j$. Then, for $k = j+1$, we have
    $$
    \cR^{j+1} = \cR^j *\cR = \bZ_{j} \bZ_1.
    $$
    Denote $[\bZ_{j} \cdot \bZ_1]_i$ as the $i$-th column of $\bZ_j \cdot \bZ_1$. To establish the result, we only need to demonstrate that the two matrices, $\cR^{k+1}$ and $\bZ_{k+1}$, have the same column entries. For $i=1$,
    $$
    \begin{aligned}
        [\bZ_{j} \bZ_1]_1 = \bZ_{j} [\bZ_1]_1 = \sum_{i=1}^{B+1} [\bZ_j]_i 
        = \sum_{i=1}^{B+1} \bfe_{\cI_{i,j}} =  \bfe_{\cup_{i=1}^{B+1} \cI_{i,j} } = \bfe_{\cI_{1,j+1}}.
    \end{aligned}
    $$
    For $i>1$, we have
    $$
    \begin{aligned}
        & [\bZ_{j} \bZ_1]_i= \bZ_{j} [\bZ_1]_i = \sum_{m\in \cI_{i,1}}\brk{\bZ_j}_{m} = \sum_{m=(i-1)B+2}^{iB+1}[\bZ_{j}]_{m}\\
        = & \sum_{m=(i-1)B}^{iB-1} \bfe_{\cI_{m,j}} = \bfe_{\cup_{m=(i-1)B}^{iB-1}\cI_{m,j} } =  \bfe_{\cI_{i,j+1}}.
    \end{aligned}
    $$
    Thus, we conclude that $\cR^{k+1} = \bZ_{k+1}$.
\end{proof}
\cref{l.R0} below reveals that when the power $k$ exceeds $d$, $\cR^k$ transforms into a matrix where the first column is entirely composed of ones, while all the other columns consist of zeros. 
\begin{corollary}
    \label{l.R0}
    For $k = d$, we have 
    $$
    \cR^k - \frac{1}{n}\mone \bu^{\T} = \mathbf{0}
    $$
    where $\mathbf{0}$ is the matrix with all entries equal 0.
\end{corollary}
\begin{proof}
    From \cref{l.R}, we have for the $i$-th column of $\cR^k - \frac{1}{n}\mone \bu^{\T}$ that
    $$
    \left[\cR^k - \frac{1}{n}\mone \bu^{\T}\right]_i = \left\{ \begin{aligned}
        - \bfe_{\frac{B^{k+1}-1}{B-1}+1:n} & \  i=1\\
        \bfe_{\cI_{i,k}} & \  i>1
    \end{aligned}\right.
    $$
    For $k=d$, the first $n$ elements of all the columns remain $0$, which implies the desired result.
\end{proof}

Now, we are ready to prove \cref{l.R0knorm}:
\begin{proof}[Proof of \cref{l.R0knorm}]
For any integer $k\le d-1$, define 
$$
n_0 := \left\lfloor \frac{n-\frac{B^{k+1} - 1}{B-1}}{B^k} \right\rfloor.
$$
This ensures that $ \frac{B^{k+1} - 1}{B-1} + n_0 B^k \le n$ and $\frac{B^{k+1} - 1}{B-1} + (n_0 + 1) B^k > n$, so that only the first $(n_0 +2)$-th columns of $\cR^k$ consist of non-zero elements.
Note that
$$
\max_{\|\bx\|_2 = 1} \left\{ \|\prt{\cR^k - \frac{1}{n}\mone\bu^{\T}} \bx\|_2^2 \right\} = \max_{\bx}\left\{\frac{\|\prt{\cR^k - \frac{1}{n}\mone\bu^{\T}} \bx\|_2^2 }{ \|\bx\|^2_2} \right\}.
$$
Then, we focus on the non-zero elements of the matrix $\cR^k - \frac{1}{n}\mone\bu^{\T}$.
$$
\begin{aligned}
    & \frac{\|\prt{\cR^k - \frac{1}{n}\mone\bu^{\T}} \bx\|_2^2 }{ \|\bx\|^2_2} &= \frac{ 1 }{ \|\bx\|^2_2} \prt{\sum_{j=2}^{n_0 + 1} B^k (x_j - x_1)^2 + \brk{n- \frac{B^{k+1} - 1}{B-1} - B^k n_0}\prt{x_{n_0 + 2} - x_1}^2} \\
    & := \frac{\tilde{\bx}^{\T} \Sigma \tilde{\bx}}{\|\bx\|^2},
\end{aligned}
$$
where $\tilde{\bx} = \prt{x_1, \cdots, x_{n_0+2}}$ is the truncated $\bx$, and 
$$
\Sigma = \prt{\begin{matrix}n - \frac{B^{k+1} - 1}{B-1} & -B^k  & \cdots & -\brk{n- \frac{B^{k+1} - 1}{B-1} - B^k n_0}\\
-B^k & B^k &  &  \\
\vdots&  & \ddots &\\
-\brk{n- \frac{B^{k+1} - 1}{B-1} - B^k n_0} & & & \brk{n- \frac{B^{k+1} - 1}{B-1} - B^k n_0}\end{matrix}},
$$
where the unspecified elements are all zeros. Since $\Sigma$ is symmetric, all the eigenvalues are real. We show by contradiction that any eigenvalue $\lambda$ of $\Sigma$ is upper bounded by $n$.
Otherwise, if there exists $\lambda > n$, we denote $\bx$ as the corresponding eigenvector of $\lambda$. Then, we have from $\Sigma\bx = \lambda\bx$ that
$$
\begin{aligned}
    \lambda x_1 &= \prt{n - \frac{B^{k+1} - 1}{B-1}}x_1 - B^k x_2 - \cdots - \brk{n- \frac{B^{k+1} - 1}{B-1} - B^k n_0}x_{n_0 + 2} \\
    \lambda x_2 &= -B^k x_1 + B^k x_2 \\
    \lambda x_3 &= -B^k x_1 + B^k x_3 \\
    & \cdots \\
    \lambda x_{n_0 + 2} &= -\brk{n- \frac{B^{k+1} - 1}{B-1} - B^k n_0}x_1 + \brk{n- \frac{B^{k+1} - 1}{B-1} - B^k n_0}x_{n_0 + 2}.
\end{aligned}
$$
Without loss of generality, assume $x_1 \neq 0$. Then, by substituting the other relations into the first one, we have
$$
\lambda = n - \frac{B^{k+1} - 1}{B-1} + \sum_{i=1}^{n_0}\frac{ B^{2k}}{\lambda - B^k} + \frac{\brk{n- \frac{B^{k+1} - 1}{B-1} - B^k n_0}^2}{\lambda - \brk{n- \frac{B^{k+1} - 1}{B-1} - B^k n_0}}.
$$
With the fact that $\lambda > n$, there holds
$$
\begin{aligned}
    \lambda \le & n - \frac{B^{k+1} - 1}{B-1} + \frac{ n_0 B^{2k}}{n - B^k} + \frac{\brk{n- \frac{B^{k+1} - 1}{B-1} - B^k n_0}^2}{n - \brk{n- \frac{B^{k+1} - 1}{B-1} - B^k n_0}}  \\
    = & n - \frac{B^{k+1} - 1}{B-1} + \frac{ n_0 B^{2k}}{n - B^k} + \frac{n^2}{ \frac{B^{k+1} - 1}{B-1} + B^k n_0} - 2n + \frac{B^{k+1} - 1}{B-1} + B^k n_0 \\
    = &  \frac{nB^k n_0}{n - B^k} + \frac{ n\prt{n- \frac{B^{k+1} - 1}{B-1} - B^k n_0 }}{ \frac{B^{k+1} - 1}{B-1} + B^k n_0} \\
    \le & \frac{nB^k n_0}{n - B^k} +  \frac{n}{n-B^k}\prt{n- \frac{B^{k+1} - 1}{B-1} - B^k n_0 } \\
    = & n\frac{ n- \frac{B^{k+1} - 1}{B-1} }{n-B^k}  < n,
\end{aligned}
$$
which is a contradiction. Thus, we have $\lambda \le n$.
It follows that
$$
\frac{\tilde{\bx}^{\T} \Sigma \tilde{\bx}}{\|\bx\|^2} \le \frac{\tilde{\bx}^{\T} \Sigma \tilde{\bx}}{\|\tilde{\bx}\|^2}  \le \lambda_{\max}(\Sigma) \le n.
$$
From the fact that the square root function is monotonically increasing on $[0,\infty)$, we have
$$
 \|\cR^k - \frac{1}{n}\mone\bu^{\T}\|_2^2= \max_{\|\bx\|_2 = 1} \left\{ \|\prt{\cR^k - \frac{1}{n}\mone\bu^{\T}} \bx\|_2^2 \right\} \le n,
$$
which implies that $\|\cR^k - \frac{1}{n}\mone\bu^{\T}\|_2 \le \sqrt{n}$ for $k\le d-1$ and $\|\cR^k - \frac{1}{n}\mone\bu^{\T}\|_2 = 0$ otherwise by \cref{l.R0}. 
\end{proof}

The transformations described in \cref{l.pp1} below are straightforward.
\begin{corollary}
    \label{l.pp1}
    For any integer $m>0$, we have
    $$
    \begin{aligned}
        \bPi_{\bu}\cR & = \bPi_{\bu}\prt{\cR - \frac{1}{n}\mone\bu^{\T}} = \prt{\cR - \frac{1}{n}\mone\bu^{\T}} \bPi_{\bu}, \\
        \bPi_{\bu}\cR^m & = \bPi_{\bu}\prt{\cR^m - \frac{1}{n}\mone\bu^{\T}} = \bPi_{\bu}\prt{\cR - \frac{1}{n}\mone\bu^{\T}}^m = \prt{\cR - \frac{1}{n}\mone\bu^{\T}}^m \bPi_{\bu}.
    \end{aligned}
    $$
\end{corollary}
To simplify the convergence analysis, we introduce the matrix $\bA_i$ defined as follows:
$$ 
\bA_i = \cC^{i} - \cC^{i-1}
$$
for $i=1,2,\cdots, d$. Specifically, $\bA_1 = \cC- \bI$. Consequently, \cref{l.Ai} below can be directly deduced from \cref{l.R} and \cref{l.pp1}.
\begin{corollary}
    \label{l.Ai}
    For $i = 1,\cdots, d$, we have
    $$
    \prt{\frac{\bu^{\T}}{n}-\mone^{\T}} \bA_i = \left\{ \begin{aligned}
        \bfe_{\frac{B^{i}-1}{B-1} + 1:\frac{B^{i+1}-1}{B-1}}^{\T} & \quad i\le d-1 \\
        \bfe_{\frac{B^{d}-1}{B-1} + 1:n}^{\T} & \quad i = d
    \end{aligned}\right. 
    $$
\end{corollary}
Intuitively, \cref{l.Ai} illustrates that $\left(\frac{\bu^{\T}}{n}-\mone^{\T}\right) \bA_i$ serves as an indicator vector representing the $(i+1)$-th layer of the graph.

\subsection{Supporting Inequalities and Lemmas}
\cref{l.sum_matrix} and \cref{l.matrix} below are frequently employed for bounding the norms of matrix summations and multiplications. Their proofs rely on the Cauchy-Schwartz inequality and the definitions of matrix norms $\|\cdot\|_2$ and $\|\cdot\|_F$.
\begin{lemma}
    \label{l.sum_matrix}
    For an arbitrary set of $m$ matrices $\{ \bA_i\}_{i=1}^m$ with the same size, we have
    $$
    \norm{ \sum_{i=1}^m \bA_i}_F^2 \le m \sum_{i=1}^m  \norm{\bA_i}_F^2.
    $$
\end{lemma}
\begin{proof}
    By the definition of Frobenius norm, we have
    $$
    \norm{\sum_{i=1}^m \bA_i}_F \le \sum_{i=1}^m \norm{\bA_i}_F.
    $$
    Taking squares on both sides and invoking the Cauchy-Schwarz inequality, we have
    $$
    \norm{\sum_{i=1}^m \bA_i}_F^2 \le \prt{\sum_{i=1}^m \norm{\bA_i}_F}^2 \le m \sum_{i=1}^m \norm{\bA_i}_F^2.
    $$
\end{proof}
\begin{lemma}
    \label{l.matrix}
    Let $\bA$, $\bB$ be two real matrices whose sizes match. Then,
    $$
    \norm{\bA\bB}_F \le \norm{\bA}_2\norm{\bB}_F.
    $$
\end{lemma}
\begin{proof}
    Let $\bA = \bU\Sigma\bV^H$ be the singular value decomposition of $\bA$, with the largest singular value $\sigma_{\max}$ and hence $\norm{\bA}_2 = \sigma_{\max}$. Then, we have
    $$
    \begin{aligned}
        \norm{\bA\bB}_F^2 & = \norm{\bU\Sigma\bV^H\bB}_F^2 = \text{trace}\prt{\prt{ \bU\Sigma\bV^H\bB}^H \prt{\bU\Sigma\bV^H\bB}}\\
        & = \text{trace}\prt{\prt{ \Sigma\bV^H\bB}^H \prt{\Sigma\bV^H\bB}} = \norm{ \Sigma\bV^H\bB}_F^2 \\
        & \le \sigma_{\max}^2 \| \bV^H\bB\|_F^2 = \sigma_{\max}^2 \text{trace}\prt{\bB^{\T}\bV\bV^H\bB} \\
        & = \sigma_{\max}^2 \text{trace}\prt{\bB^{\T}\bB} = \sigma_{\max}^2 \norm{\bB}_F^2 \\
        & = \norm{\bA}_2^2 \norm{\bB}_F^2,
    \end{aligned}
    $$
    which implies the desired result.
\end{proof}
 Lemma \ref{l.martin} below will be used in the last step for deriving the convergence rate of BTPP.
\begin{lemma}
    \label{l.martin}
    Let $A,B,C$ and $\alpha$ be positive constants and $T$ be a positive integer. Define function
    $$
    g(\gamma) = \frac{A}{\gamma (T+1)} + B\gamma + C\gamma^2.
    $$
    Then,
    $$
    \inf_{\gamma \in (0,\frac{1}{\alpha}]} g(\gamma) \le 2\prt{\frac{AB}{T+1}}^{\frac{1}{2}} + 2C^{\frac{1}{3}} \prt{\frac{A}{T+1}}^{\frac{2}{3}} + \frac{\alpha A}{T+1},
    $$
    where the upper bound can be achieved by choosing $\gamma = \min\crk{\prt{\frac{A}{B(T+1)}}^{\frac{1}{2}} , \prt{\frac{A}{C(T+1)}}^{\frac{1}{3}},\frac{1}{\alpha}}$.
\end{lemma}
\begin{proof}
    See Lemma 26 in \cite{koloskova2021improved} for a reference.
\end{proof}

 Lemma \ref{lem:independ_help} is a technical result related to random variables.
\begin{lemma}
    \label{lem:independ_help}
    Consider three random variables $X$, $Y$, and $Z$. Assume that $Z$ is independent with $(X, Y)$. Let $h$ and $g$ be functions such that the conditional expectation $\mathbb{E}[g(Y, Z) \mid Y] = 0$. We have
    $$
    \expect \prt{h(X) g(Y,Z)} = 0.
    $$
\end{lemma}
\begin{proof}
    It implies by the condition $Z\perp \!\!\! \perp (X,Y)$ that $\sigma(Z)\perp \!\!\!\perp \sigma(X,Y)$. Then,
    $$
    \begin{aligned}
        \expect \brk{ h(X) g(Y,Z)|Y} & = \expect \crk{ \expect\brk{h(X) g(Y,Z)|X,Y}|Y} \\
        & = \expect \crk{ h(X) \expect\brk{g(Y,Z)|X,Y}|Y}. \\ 
    \end{aligned}
    $$
    It suffices to show
    $$
    \expect\brk{g(Y,Z)|X,Y} = \expect\brk{g(Y,Z)|Y} (= 0).
    $$
    Let $f_g(y) = \expect\prt{g(y,Z)}$. Since $\sigma(Z)\perp \!\!\!\perp \sigma(X,Y)$, we have $\sigma(Z)\perp \!\!\!\perp \sigma(Y)$. Then,
    $$
    f_g(Y) = \expect\prt{g(Y,Z)|Y} = \expect\brk{g(Y,Z)|X,Y} ,
    $$
    which follows directly from (10.17) in \cite{resnick2019probability}.
    Thus, by the Tower Rule, we reach the statement as follows:
    $$
    \expect \prt{h(X) g(Y,Z)} = \expect\crk{\expect \prt{h(X) g(Y,Z)|Y} } = 0.
    $$
\end{proof}

\subsection{Proofs of Key Lemmas}

\label{pf.thm:convergence}
In this section, we prove several key lemmas for proving the main convergence result of BTPP.

\subsubsection{Preparation}
\cref{alg:pp}, as encapsulated by the equations in \eqref{eq:ppSGD}, can be succinctly expressed in the following matrix form:
\begin{equation}
    \label{eq:matrixpp}
    \left( \begin{array}{c}
         \bX^{(t+1)}\\
        \bY^{(t+1)}  
    \end{array}\right) = 
    \left( \begin{array}{cc}
         \cR & -\gamma \cR\\
        \mathbf{0} & \cC  
    \end{array}\right)\left( \begin{array}{c}
         \bX^{(t)}\\
        \bY^{(t)}  
    \end{array}\right) + 
    \left( \begin{array}{c}
         \mathbf{0}\\
        \bG^{(t+1)} -\bG^{(t)} 
    \end{array}\right).
\end{equation}
By repeatedly applying equation \eqref{eq:matrixpp} starting from time step $t$ and going back to time step $0$, we arrive at the following relation:
$$
\begin{aligned}
    \left( \begin{array}{c}
         \bX^{(t)}\\
        \bY^{(t)}  
    \end{array}\right) 
    & = \left( \begin{array}{cc}
         \cR & -\gamma \cR\\
        \mathbf{0} & \cC  
    \end{array}\right)^{t}\left( \begin{array}{c}
         \bX^{(0)}\\
        \bY^{(0)}  
    \end{array}\right) + 
    \sum_{m=0}^{t-1}
    \left( \begin{array}{cc}
         \cR & -\gamma \cR\\
        \mathbf{0} & \cC  
    \end{array}\right)^{t-m-1}\left( \begin{array}{c}
         \mathbf{0}\\
        \bG^{(m+1)} -\bG^{(m)} 
    \end{array}\right).
\end{aligned}
$$
For the sake of clarity, we start with introducing some simple definitions. Any matrix raised to the power of 0 is defined as the identity matrix $\bI$, which matches the original matrix in dimension. The only exceptions are $\prt{\cR - \frac{1}{n}\mone\bu^{\T}}^0 := \bPi_{\bu}$ and $\prt{\cC - \frac{1}{n}\bv\mone^{\T}}^0 := \bPi_{\bv}$ for convenience. Furthermore, we introduce the following terms:
$$
\Bar{\bX}^{(t)} := \frac{1}{n}\mone\bu^{\T}\bX^{(t)},\ \Bar{\bY}^{(t)}:= \frac{1}{n}\bv\mone^{\T}\bY^{(t)}.
$$
Note that, for any given integer $m>0$, 
$$
\left( \begin{array}{cc}
         \cR & -\gamma \cR\\
        \mathbf{0} & \cC  
    \end{array}\right)^m = \left( \begin{array}{cc}
         \cR^m & -\gamma \sum_{j=1}^{m}\cR^j\cC^{m-j}\\
        \mathbf{0} & \cC^m  
    \end{array}\right).
$$
As a result, given the initial condition $\bY^{(0)} = \bG^{(0)}$, we can deduce the outcomes of $\bX^{(t)}$ and $\bY^{(t)}$ as follows.
\begin{align}
    \bX^{(t)} & = \cR^{t}\bX^{(0)} - \gamma \sum_{m=0}^{t-2} \sum_{j=1}^{t-m-1}\cR^j\cC^{t-m-1-j} \brk{\bG^{(m+1)} - \bG^{(m)}}  -\gamma \sum_{j=1}^{t}\cR^j\cC^{t-j} \bG^{(0)}, \label{eq:pp1} \\
    \bY^{(t)} & = \sum_{m=0}^{t-1} \cC^{t-m-1} \brk{\bG^{(m+1)} - \bG^{(m)}} + \cC^{t} \bG^{(0)}. \label{eq:pp2} 
\end{align}
Then, after multiplying $\bPi_{\bu}$ and $\bPi_{\bv}$ to equation \eqref{eq:pp1} and equation \eqref{eq:pp2} respectively, and invoking \cref{l.pp1}, we have
\begin{align}
    \begin{split}
        \bPi_{\bu}\bX^{(t)} & = \prt{\cR - \frac{1}{n}\mone\bu^{\T}}^{t}\bX^{(0)} - \gamma \sum_{j=1}^{\min\crk{d-1,t}}\prt{\cR  - \frac{1}{n}\mone\bu^{\T}}^j\cC^{t-j} \bG^{(0)}\\
        &\quad - \gamma \sum_{m=0}^{t-2} \sum_{j=1}^{\min\crk{t-m-1, d - 1}}\prt{\cR - \frac{1}{n}\mone\bu^{\T} }^j\cC^{t-m-1-j} \brk{\bG^{(m+1)} - \bG^{(m)}},
      \end{split} \label{eq:pp11} \\
    \begin{split}
        \bPi_{\bv}\bY^{(t)} & = \sum_{m=\max\crk{0,t-d}}^{t-1} \prt{\cC - \frac{1}{n}\bv\mone^{\T}}^{t-m-1} \brk{\bG^{(m+1)} - \bG^{(m)}} + \prt{\cC - \frac{1}{n}\bv\mone^{\T}}^{t} \bG^{(0)} \\
        & = \sum_{m=0}^{ \min\{t,d\} - 1} \prt{\cC - \frac{1}{n}\bv\mone^{\T}}^{m} \prt{\bG^{(t-m)} - \bG^{(t-m-1)}} + \prt{\cC - \frac{1}{n}\bv\mone^{\T}}^{t} \bG^{(0)} \\
        & = \sum_{m=1}^{\min\{t,d\}}\bA_m \bG^{(t-m)} + \bPi_{\bv}\bG^{(t)}.
    \end{split}
    \label{eq:pp22} 
\end{align}

\subsubsection{Analysis of the Variance}
Denote by $\cF_t$ the $\sigma$-algebra generated by $\{\bxi_0, \cdots, \bxi_{t-1}\}$, and define $\expect\brk{\cdot|\cF_t}$ as the conditional expectation given $\cF_{t}$. \cref{l.bound_var} provides an estimate for the variance of the gradient estimator $G(\bX^{(t)}, \bxi^{(t)})$. 
\begin{lemma}
    \label{l.bound_var}
    Under \cref{a.var}, for any given power $k\le d-1$, we have for all $t\ge 0$ that
    $$
    \begin{aligned}
        \expect \brk{\norm{\prt{ \cC - \frac{1}{n}\bv\mone^{\T}}^k \prt{\bG(\bX^{(t)}, \bxi^{(t)}) - \nabla\bF(\bX^{(t)})} }_F^2 \mid \cF_t } 
        \le 2n\sigma^2.
    \end{aligned}
    $$
\end{lemma}

\label{pf.l.bound_var}
\begin{proof}
    For any given $t$ and $i\ne j$, due to the independently drawn sample $\xi_i^{(t)}$, we have that $\xi_i^{(t)}$ is independent of $(\cF_t, \xi_j^{(t)})$, and thus $\xi_i^{(t)}$ is independent of $\sigma(x_i^{(t)}, x_j^{(t)}, \xi_j^{(t)})$. Hence, invoking \cref{lem:independ_help} and \cref{a.var} yields
    $$
    \begin{aligned}
    & \expect\brk{ \nabla F(x_i^{(t)};\xi_i^{(t)}) - \nabla f_i(x_i^{(t)}) \bigg| x_i^{(t)}} = \expect\brk{ \nabla F(x_i^{(t)};\xi_i^{(t)}) - \nabla f_i(x_i^{(t)}) \bigg|  \cF_t } = 0,\\
    &    \expect \left\langle \nabla F(x_i^{(t)};\xi_i^{(t)}) - \nabla f_i(x_i^{(t)}), \nabla F(x_j^{(t)};\xi_j^{(t)}) - \nabla f_j(x_j^{(t)}) \right\rangle = 0.
    \end{aligned}
    $$
Then, for any index set $\cI \subseteq \{1,2,\cdots,n\}$, we have $\expect \|\bfe_{\cI}^{\T} \prt{\bG^{(t)} - \nabla \bF(\bX^{(t)})}\|_2^2 \le |\cI| \sigma^2$.

 Notice that
$$
\begin{aligned}
    \norm{\prt{ \cC - \frac{1}{n}\bv\mone^{\T}}^k \prt{\bG^{(t)} - \nabla\bF(\bX^{(t)})} }_F^2 = & \norm{ \bfe_{\frac{B^{k+1}-1}{B-1}+1:n} \prt{\bG^{(t)} - \nabla \bF(\bX^{(t)})}}_2^2 \\
    & + \sum_{i=2}^n\norm{ \bfe_{\cI_{i,k}}\prt{\bG^{(t)} - \nabla \bF(\bX^{(t)})}}_2^2.
\end{aligned}
$$
Thus, we obtain the desired result by invoking \cref{l.R} and \cref{l.R0} after taking expectation on both sides of the above relation:
$$
\expect \norm{\prt{ \cC - \frac{1}{n}\bv\mone^{\T}}^j \prt{\bG^{(t)} - \nabla\bF(\bX^{(t)})} }_F^2 \le  2\prt{n-\frac{B^{k+1}-1}{B-1}} \sigma^2 \le 2n\sigma^2.
$$
\end{proof}
Under \cref{a.var} and the randomly selected samples, \cref{l.bound_var} and \cref{co:Ai} below provide an initial estimation for the variance terms. 
The proof of \cref{co:Ai} is directly from the analysis in \cref{pf.l.bound_var} and \cref{l.Ai}.
\begin{corollary}
    \label{co:Ai}
    Under \cref{a.var}, we have for all $t\ge 0$ that
    $$
    \sum_{k=1}^{d}\expect\|\prt{\frac{\bu^{\T}}{n}-\mone^{\T}} A_k \prt{\bG(\bX^{(t)}, \bxi^{(t)}) - \nabla \bF(\bX^{(t)})}\|_2^2 \le (n-1)\sigma^2.
    $$
\end{corollary}

\subsubsection{Proof of Lemma \ref{l.X_diff}}

\label{pf.l.X_diff}

\begin{proof}
    Notice that
    $$
    \begin{aligned}
        & \bar{\bX}^{(t+1)} - \Bar{\bX}^{(t)} = -\gamma\frac{1}{n}\mone\bu^{\T} \bY^{(t)} = -\gamma\frac{1}{n}\mone\bu^{\T} \brk{\bPi_{\bv}\bY^{(t)} + \frac{1}{n}\bv\mone^{\T}\bY^{(t)}} \\
        & \qquad = -\gamma\frac{1}{n}\mone\bu^{\T} \bPi_{\bv}\bY^{(t)}  -\gamma \mone\mone^{\T} \bY^{(t)}  = -\gamma\mone \prt{\frac{\bu^{\T}}{n}-\mone^{\T}}\bPi_{\bv}\bY^{(t)}  -\gamma \mone\mone^{\T} \bY^{(t)}.
    \end{aligned}
    $$
    Multiplying $\mone^{\T}$ on both sides of equation \eqref{eq:pp2}, we have $\mone^{\T}\bY^{(t)} = \mone^{\T}\bG^{(t)}$ for any integer $t$. Thus, in light of equation \eqref{eq:pp22}, we have
    $$
    \begin{aligned}
        &\bar{\bX}^{(t+1)} - \Bar{\bX}^{(t)}  = -\gamma\mone \prt{\frac{\bu^{\T}}{n}-\mone^{\T}} \bPi_{\bv}\bY^{(t)}  -\gamma \mone\mone^{\T} \bG^{(t)}  \\
        & \qquad = -\gamma\mone \prt{\frac{\bu^{\T}}{n}-\mone^{\T}}\sum_{m=1}^{\min\crk{t,d}} \bA_m\bG^{(t-m)} - \gamma \mone  \prt{\frac{\bu^{\T}}{n}-\mone^{\T}} \bG^{(t)}  -\gamma \mone\mone^{\T} \bG^{(t)} \\
        & \qquad = -\gamma \mone \prt{\frac{\bu^{\T}}{n}-\mone^{\T}}\sum_{m=1}^{\min\crk{t,d}} \bA_m\prt{\bG^{(t-m)} - \nabla \bF(\bX^{(t-m)})} \\
        & \qquad \qquad - \gamma \mone \prt{\frac{\bu^{\T}}{n}-\mone^{\T}}\sum_{m=1}^{\min\crk{t,d}} \bA_m \nabla \bF(\bX^{(t-m)})  - \gamma \mone\frac{\bu^{\T}}{n} \bG^{(t)}  \\
        & \qquad = -\gamma \mone \prt{\frac{\bu^{\T}}{n}-\mone^{\T}}\sum_{m=1}^{\min\crk{t,d}} \bA_m\prt{\bG^{(t-m)} - \nabla \bF(\bX^{(t-m)})} \\
        & \qquad \qquad -\gamma \mone \prt{\frac{\bu^{\T}}{n}-\mone^{\T}} \sum_{m=\max\{0,t-d\}}^{t-1} \prt{\cC - \frac{1}{n}\bv\mone^{\T}}^{t-m-1} \brk{\nabla \bF(\bX^{(m+1)}) - \nabla \bF(\bX^{(m)})} \\
        & \qquad \qquad - \gamma \mone \prt{\frac{\bu^{\T}}{n}-\mone^{\T}} \prt{\cC - \frac{1}{n}\bv\mone^{\T}}^{t}\nabla\bF(\bX^{(0)}) + \gamma \mone \prt{\frac{\bu^{\T}}{n}-\mone^{\T}} \nabla \bF(\bX^{(t)}) - \gamma \mone\frac{\bu^{\T}}{n} \bG^{(t)}.
    \end{aligned}
    $$
    Hence, taking the F-norm and expectation on both sides, we have from \cref{l.sum_matrix} that
    \begin{equation}
        \label{eq:diff_X_1}
        \begin{aligned}
            & \expect\|\bar{\bX}^{(t+1)} - \Bar{\bX}^{(t)} \|_F^2 \le 5\gamma^2 n \expect \norm{\sum_{m=1}^{\min\crk{t,d}} \prt{\frac{\bu^{\T}}{n}-\mone^{\T}}\bA_m\prt{\bG^{(t-m)} - \nabla \bF(\bX^{(t-m)})}}_F^2 \\
            & \qquad + 5\gamma^2 n \expect \norm{\prt{\frac{\bu^{\T}}{n}-\mone^{\T}}\sum_{m=\max\{0,t-d\}}^{t-1} \prt{\cC - \frac{1}{n}\bv\mone^{\T}}^{t-m-1} \brk{\nabla \bF(\bX^{(m+1)}) - \nabla \bF(\bX^{(m)})}}_F^2 \\
            & \qquad + 5\gamma^2 n \expect \norm{\prt{\frac{\bu^{\T}}{n}-\mone^{\T}}\prt{\cC - \frac{1}{n}\bv\mone^{\T}}^{t}\nabla\bF(\bX^{(0)})}_F^2 \\
            &\qquad  + 5\gamma^2 n \expect \norm{\frac{\bu^{\T}}{n} \prt{\nabla \bF(\bX^{(t)}) - \bG^{(t)}}}_F^2 + 5\gamma^2 n \expect \norm{\mone^{\T}\bF(\bX^{(t)})}_F^2.
        \end{aligned}
    \end{equation}
    Note that, invoking \cref{l.bound_var} and \cref{co:Ai} yields
    $$
    \begin{aligned}
        & \expect \norm{\sum_{m=1}^{\min\crk{t,d}} \prt{\frac{\bu^{\T}}{n}-\mone^{\T}}\bA_m\prt{\bG^{(t-m)} - \nabla \bF(\bX^{(t-m)})}}_F^2 \\
        = & \sum_{m=1}^{ \min\crk{t,d} } \expect\norm{\prt{\frac{\bu^{\T}}{n}-\mone^{\T}}\bA_m\prt{\bG^{(t-m)} - \nabla \bF(\bX^{(t-m)})}}_F^2 \\
        \le & (n-1)\sigma^2.
    \end{aligned}
    $$
    From \cref{a.smooth}, \cref{l.matrix} and \cref{l.sum_matrix}, we have
    $$
    \begin{aligned}
        & \expect \norm{\prt{\frac{\bu^{\T}}{n}-\mone^{\T}}\sum_{m=\max\{0,t-d\}}^{t-1} \prt{\cC - \frac{1}{n}\bv\mone^{\T}}^{t-m-1} \brk{\nabla \bF(\bX^{(m+1)}) - \nabla \bF(\bX^{(m)})}}_F^2 \\
        \le & d\sum_{m=\max\{0,t-d\}}^{t-1} \expect \norm{ \prt{\frac{\bu^{\T}}{n}-\mone^{\T}}\prt{\cC - \frac{1}{n}\bv\mone^{\T}}^{t-m-1} \brk{\nabla \bF(\bX^{(m+1)}) - \nabla \bF(\bX^{(m)})}}_F^2 \\
        \le & d \sum_{m=\max\{0,t-d\}}^{t-1} \expect \norm{ \prt{\frac{\bu^{\T}}{n}-\mone^{\T}}\prt{\cC - \frac{1}{n}\bv\mone^{\T}}^{t-m-1} }_2^2 \norm{ \nabla \bF(\bX^{(m+1)}) - \nabla \bF(\bX^{(m)})}_F^2 \\
        \le & nd L^2 \sum_{m=\max\{0,t-d\}}^{t-1} \expect \norm{\bX^{(m+1)} - \bX^{(m)}}_F^2 \\
        \le & nd L^2 \sum_{m=\max\{0,t-d\}}^{t-1} 3\prt{\expect \norm{\bX^{(m+1)} - \bar{\bX}^{(m+1)}}_F^2 + \expect \norm{\bX^{(m)} - \bar{\bX}^{(m)}}_F^2 + \expect \norm{\bar{\bX}^{(m+1)} - \bar{\bX}^{(m)}}_F^2 }.
    \end{aligned}
    $$
    Thus, summing over $t$ in \eqref{eq:diff_X_1} from 0 to $T$, combining all the inequalities above, and invoking \cref{a.var} and \cref{a.smooth}, we have
    \begin{equation}
    \label{eq:refine-X-diff}
    \begin{aligned}
        & \sum_{t=0}^{T} \expect\|\bar{\bX}^{(t+1)} - \Bar{\bX}^{(t)} \|_F^2 \le 5\gamma^2 n(n-1)\sigma^2(T+1) + 30\gamma^2n^2d^2 L^2 \sum_{t=0}^{T}\expect \norm{\bPi_{\bu}\bX^{(t)}}_F^2 \\
        &\qquad + 15 \gamma^2n^2d^2 L^2 \sum_{t=0}^{T} \expect \norm{\bar{\bX}^{(t+1)} - \bar{\bX}^{(t)}}_F^2 \\
        &\qquad + 5\gamma^2 n^2 \sum_{t=0}^{\min\{t,d-1\}} \expect \norm{\nabla\bF(\bX^{(0)})}_F^2  + 5\gamma^2 n \sigma^2 (T+1) \\
        &\qquad + 5\gamma^2 n \sum_{t=0}^{T}\prt{2\expect\norm{\mone^{\T}\nabla \bF(\bX^{(t)}) - \mone^{\T}\nabla\bF(\bar{\bX}^{(t)}) }_2^2+2n^2\expect\norm{\frac{1}{n}\mone^{\T}\nabla\bF(\bar{\bX}^{(t)})}_2^2}  \\
        &\quad \le  5\gamma^2 n^2\sigma^2(T+1) + 40\gamma^2n^2d^2L^2 \sum_{t=0}^{T}\expect \norm{\bPi_{\bu}\bX^{(t)}}_F^2 + 15 \gamma^2n^2d^2L^2 \sum_{t=0}^{T} \expect \norm{\bar{\bX}^{(t+1)} - \bar{\bX}^{(t)}}_F^2 \\
        &\qquad + 5\gamma^2 n^2 d \norm{\nabla\bF(\bX^{(0)})}_F^2  + 10\gamma^2 n^3\sum_{t=0}^{T}\expect\norm{\nabla f(x_1^{(t)})}_2^2.
    \end{aligned}
    \end{equation}
    Since $\gamma \le \frac{1}{10ndL}$, we have $15 \gamma^2n^2d^2L^2\le \frac{1}{6}$, and the desired result follows.
\end{proof}

\subsubsection{Proof of Lemma \ref{l.PiX}}
\label{pf.l.PiX}
\begin{proof}
    
We show the upper bound for  $\expect\norm{\bPi_{\bu}\bX^{(t)}}_F^2$ by studying equation \eqref{eq:pp11} and bound the F-norm of each term respectively. From \cref{l.R0}, we can change the power of $\cR - \frac{1}{n}\mone\bu^{\T}$ to at most $d-2$:
\begin{equation}
    \begin{aligned}
        \bPi_{\bu}\bX^{(t)}  = &  (\cR^t-\frac{1}{n}\mone\bu^{\T})\bX^{(0)} - \gamma \sum_{m=0}^{t-2} \sum_{j=1}^{\min\{t-m-1, d-1\}}(\cR-\frac{1}{n}\mone\bu^{\T})^j\cC^{t-m-1-j} \brk{\bG^{(m+1)} 
        - \bG^{(m)}} \\
        & \qquad -\gamma  \sum_{j=1}^{\min\{t,d-1\}}\prt{\cR - \frac{1}{n}\mone\bu^{\T}}^j\cC^{t-j} \bG^{(0)}.
    \end{aligned}
    \label{eq:ppbX0}
\end{equation}
Then, we derive the following decomposition by pairing the gradients with each of the stochastic gradients in order to use \cref{a.var}.
\begin{equation}
    \label{eq:ppbX}
    \begin{aligned}
        \bPi_{\bu}\bX^{(t)} =&  (\cR^{\T}-\frac{1}{n}\mone\bu^{\T})\bX^{(0)} -\gamma  \sum_{j=1}^{\min\{t,d-1\}}\prt{\cR - \frac{1}{n}\mone\bu^{\T}}^j\prt{\cC-\frac{1}{n}\bv\mone^{\T}}^{t-j} \prt{\bG^{(0)} - \nabla \bF(\bX^{(0)})}  \\
        & - \gamma  \sum_{j=1}^{\min\{t,d-1\}}\prt{\cR - \frac{1}{n}\mone\bu^{\T}}^j \frac{1}{n}\bv\mone^{\T} \prt{\bG^{(0)} - \nabla \bF(\bX^{(0)})} \\
        & - \gamma \sum_{j=1}^{\min\{t,d-1\}}\prt{\cR - \frac{1}{n}\mone\bu^{\T}}^j\prt{\cC-\frac{1}{n}\bv\mone^{\T}}^{t-j}  \nabla \bF(\bX^{(0)}) \\
        & - \gamma \sum_{j=1}^{\min\{t,d-1\}}\prt{\cR - \frac{1}{n}\mone\bu^{\T}}^j \frac{1}{n}\bv\mone^{\T}  \nabla \bF(\bX^{(0)}) \\
        &- \gamma \sum_{m=0}^{t-2} \sum_{j=\max\{1, t-m-d\}}^{\min\{t-m-1, d-1\}}(\cR-\frac{1}{n}\mone\bu^{\T})^j \prt{\cC - \frac{1}{n}\bv\mone^{\T}}^{t-m-1-j} \\
        &\qquad \qquad \qquad \qquad \brk{\bG^{(m+1)} - \nabla \bF(\bX^{(m+1)}) + \nabla \bF(\bX^{(m)})- \bG^{(m)}} \\
        &  - \gamma \sum_{m=0}^{t-2} \sum_{j=\max\{1, t-m-d\}}^{\min\{t-m-1, d-1\}}(\cR-\frac{1}{n}\mone\bu^{\T})^j\prt{\cC - \frac{1}{n}\bv\mone^{\T}}^{t-m-1-j}  \brk{\nabla \bF(\bX^{(m+1)}) - \nabla \bF(\bX^{(m)})} \\
        &  - \gamma \sum_{m=0}^{t-2} \sum_{j=1}^{\min\{t-m-1, d-1\}} \prt{ \cR-\frac{1}{n}\mone\bu^{\T}}^j \frac{\bv\mone^{\T}}{n}\brk{\bG^{(m+1)} - \nabla \bF(\bX^{(m+1)}) + \nabla \bF(\bX^{(m)})- \bG^{(m)}} \\
        & - \gamma\sum_{m=0}^{t-2} \sum_{j=1}^{\min\{t-m-1, d-1\}}\prt{ \cR-\frac{1}{n}\mone\bu^{\T}}^j\frac{\bv\mone^{\T}}{n}\brk{\nabla \bF(\bX^{(m+1)}) - \nabla \bF(\bX^{(m)})} \\
        := &\prt{\cR-\frac{1}{n}\mone\bu^{\T}}^{t}\bX^{(0)} 
        - \gamma \bQ_{0,t,1}- \gamma \bQ_{0,t,2}- \gamma \bQ_{0,t,3}- \gamma \bQ_{0,t,4}\\
        & - \gamma \bQ_{1,t} - \gamma \bQ_{2,t} -  \gamma \bQ_{3,t} - \gamma \bQ_{4,t},
    \end{aligned}
\end{equation}
where the terms from $\bQ_{0,t,1}$ to $\bQ_{0,t,4}$ and from $\bQ_{1,t}$ to $\bQ_{4,t}$ correspond to each term following $\prt{\cR-\frac{1}{n}\mone\bu^{\T}}^{\T}\bX^{(0)}$ one-by-one.

We assume that $d\ge2$, since $d=1$ makes the summation illegal (summing over $j$ from a positive number to a non-positive number), in which case $\bPi_{\bu}\bX^{(t)}$ degenerates to $(\cR^{t}-\frac{1}{n}\mone\bu^{\T})\bX^{(0)}$ and hence by \cref{l.R0}, there is no consensus error, i.e.
$$
\sum_{t=0}^{T} \norm{ \bPi_{\bu}\bX^{(t)}}_F^2 = 0.
$$
For $d\ge 2$, \cref{l.bXQ1} - \cref{l.bXQ3} below introduce the upper bounds for the F-norms of $\bQ_{1,t}$, $\bQ_{2,t}$ and $\bQ_{3,t}+\bQ_{0,t,2}$, summing from $t=0$ to $T$. \cref{l.bXQ0} establishes a similar upper bound for the F-norm of $\bQ_{0,t,1} +  \bQ_{0,t,3}$. Furthermore, \cref{l.bXQ00} provides the upper bound for the F-norm of $\bQ_{0,t,4} + \bQ_{4,t}$.
\begin{lemma}
    \label{l.bXQ1}
    For any iteration number $T$, we have
    $$
    \sum_{t=0}^{T}\expect\|\bQ_{1,t}\|_F^2 \le 
    32n^2 d^4 (T+1) \sigma^2.
    $$
\end{lemma}
\begin{proof}
    To make the summation legal given $d\ge 2$, we need $d-1 \ge t-m-d$, which implies that $m\ge t+1-2d$. Then,
    $$
    \begin{aligned}
        \bQ_{1,t} = & \sum_{m=0}^{t-2} \sum_{j=\max\{1, t-m-d\}}^{\min\{t-m-1, d-1\}}\prt{\cR-\frac{1}{n}\mone\bu^{\T}}^j
        \prt{\cC - \frac{1}{n}\bv\mone^{\T}}^{t-m-1-j} \\
        & \qquad \qquad \qquad \brk{\bG^{(m+1)} - \nabla \bF(\bX^{(m+1)}) + \nabla \bF(\bX^{(m)})- \bG^{(m)}} \\
        =& \sum_{m=\max\{t+1-2d,0\}}^{t-2} \sum_{j=\max\{1, t-m-d\}}^{\min\{t-m-1, d-1\}}\prt{\cR-\frac{1}{n}\mone\bu^{\T}}^j\\
        & \qquad\qquad
        \prt{\cC - \frac{1}{n}\bv\mone^{\T}}^{t-m-1-j} \brk{\bG^{(m+1)} - \nabla \bF(\bX^{(m+1)}) + \nabla \bF(\bX^{(m)})- \bG^{(m)}}.
    \end{aligned}
    $$
    Invoking \cref{l.sum_matrix}, \cref{l.R0knorm} and \cref{l.bound_var}, we have
    $$
    \begin{aligned}
        &\expect \| \bQ_{1,t} \|_F^2 \le 2(d-1)^2 \sum_{m=\max\{t+1-2d,0\}}^{t-2} \sum_{j=\max\{1, t-m-d\}}^{\min\{t-m-1, d-1\}}\expect\norm{\prt{\cR-\frac{1}{n}\mone\bu^{\T}}^j}_2^2 \cdot\\
        & \qquad\qquad \norm{\prt{\cC - \frac{1}{n}\bv\mone^{\T}}^{t-m-1-j} \brk{\bG^{(m+1)} - \nabla \bF(\bX^{(m+1)}) + \nabla \bF(\bX^{(m)})- \bG^{(m)}}}_F^2 \\
        & \quad \le 2n(d-1)^2 \sum_{m=\max\{t+1-2d,0\}}^{t-1} \sum_{j=\max\{1, t-m-d\}}^{\min\{t-m-1, d-1\}} 8n\sigma^2 \\
        & \quad \le  32n^2 d^4\sigma^2.
    \end{aligned}
    $$
    Summing over $t$ from 0 to $T$, we get the desired result.
\end{proof}

\begin{lemma}
    \label{l.bXQ2}
    For any $\gamma \le \frac{1}{20ndL}$, we have
    $$
    \begin{aligned}
        \sum_{t=0}^{T} \expect \| \bQ_{2,t} \|_F^2 & \le 52 n^2 d^4 L^2 \sum_{t=0}^{T} \expect  \norm{\bPi_{\bu}\bX^{(t)}}_F^2  + 144\gamma^2 n^4 d^4 (T+1) L^2 \sigma^2 \\
        &+ 360\gamma^2 n^5 d^4 L^2 \sum_{t=0}^{T}\expect\norm{\frac{1}{n}\nabla \bF(\bar{\bX}^{(t)})}_2^2 + 360\gamma^2 n^4d^5 L^2\norm{\nabla \bF(\bX^{(0)})}_F^2.
    \end{aligned}
    $$
\end{lemma}

\begin{proof}
    Similar to the proof of \cref{l.bXQ1}, we have
    $$
    \bQ_{2,t} =  \sum_{m=\max\{t+1-2d,0\}}^{t-1} \sum_{j=\max\{1, t-m-d\}}^{\min\{t-m-1, d-1\}}\prt{\cR-\frac{1}{n}\mone\bu^{\T}}^j
    \prt{\cC - \frac{1}{n}\bv\mone^{\T}}^{t-m-1-j} \brk{ \nabla \bF(\bX^{(m+1)}) - \nabla \bF(\bX^{(m)})}.
    $$
    Invoking \cref{l.sum_matrix} and \cref{a.smooth}, we obtain
    $$
    \begin{aligned}
        \expect \| \bQ_{2,t}\|_F^2 & \le 4 n^2 d^3  \sum_{m=\max\{t+1-2d,0\}}^{t-1} \expect \brk{ \| \nabla \bF(\bX^{(m+1)}) - \nabla \bF(\bX^{(m)}) \|_F^2}\\
        &\le 4 n^2 d^3 L^2  \sum_{m=\max\{t+1-2d,0\}}^{t-1} \expect \norm{ \bX^{(m+1)} - \bar{\bX}^{(m+1)} + \bar{\bX}^{(m)}- \bX^{(m)} + \bar{\bX}^{(m+1)} - \bar{\bX}^{(m)}}_F^2 \\
        & \le 24 n^2 d^3 L^2 \sum_{m=\max\{t+1-2d,0\}}^{t } \expect  \norm{\bPi_{\bu}\bX^{(t)}}_F^2 +  12 n^2 d^3 L^2 \sum_{m=\max\{t+1-2d,0\}}^{t-1 } \expect \norm{\bar{\bX}^{(m+1)} - \bar{\bX}^{(m)}}_F^2.
    \end{aligned}
    $$
    It follows by summing over $t$ from $0$ to $T$ and applying \cref{l.X_diff} that
    $$
    \begin{aligned}
        & \sum_{t=0}^{T} \expect \| \bQ_{2,t}\|_F^2  \le 48 n^2 d^4 L^2 \sum_{t=0}^{T} \expect  \norm{\bPi_{\bu}\bX^{(t)}}_F^2 +  24 n^2 d^4 L^2 \sum_{t=0}^{T} \expect  \norm{\bar{\bX}^{(t+1)} - \bar{\bX}^{(t)}}_F^2\\
        &\qquad \le  \prt{48 n^2 d^4 L^2 + 1200\gamma^2 n^4 d^6 L^4} \sum_{t=0}^{T} \expect  \norm{\bPi_{\bu}\bX^{(t)}}_F^2 + 144\gamma^2 n^4 d^4 (T+1) L^2 \sigma^2 \\
        & \qquad \qquad + 360\gamma^2 n^5 d^4 L^2 \sum_{t=0}^{T}\expect\norm{\nabla f(x_1^{(t)})}_2^2 + 144\gamma^2 n^4d^5 L^2 \norm{\nabla \bF(\bX^{(0)})}_F^2.
    \end{aligned}
    $$
    Hence, under the condition that $\gamma \le \frac{1}{20ndL}$, there holds $1200\gamma^2 n^4 (d-1)^6 L^4 \le 4n^2 (d-1)^4 L^2$, which implies the desired result.
\end{proof}

\begin{lemma}
    \label{l.bXQ3}
    For any $T$, we have
    $$
    \sum_{t=0}^{T}\expect \| \bQ_{3,t} + \bQ_{0,t,2} \|_F^2 \le d^2 n^2 (T+1) \sigma^2.
    $$
\end{lemma}

\begin{proof}
    By definition, we have
    $$
    \begin{aligned}
        & \bQ_{3,t} + \bQ_{0,t,2} = \sum_{j=1}^{\min\{t,d-1\}} \prt{\cR - \frac{1}{n}\mone\bu^{\T}}^j \frac{1}{n}\bv\mone^{\T}\prt{\bG^{(0)} - \nabla \bF(\bX^{(0)})} +  \\
         & \sum_{m=0}^{t-2} \sum_{j=1}^{\min\{t-m-1, d-1\}} 
        \prt{ \cR-\frac{1}{n}\mone\bu^{\T}}^j \frac{\bv\mone^{\T}}{n}\brk{\bG^{(m+1)} - \nabla \bF(\bX^{(m+1)}) + \nabla \bF(\bX^{(m)})- \bG^{(m)}} \\
        = & \sum_{m=\max\{t-d,0\}}^{t-1} \prt{ \cR-\frac{1}{n}\mone\bu^{\T}}^{t-m}\frac{\bv\mone^{\T}}{n} \brk{\bG^{(m)} - \nabla \bF(\bX^{(m)})}.
    \end{aligned}
    $$
    Thus, invoking \cref{l.sum_matrix} and \cref{l.R0knorm}, we have
    $$
    \expect \| \bQ_{3,t}\|_F^2 \le  n d \sum_{m=\max\{t-d,0\}}^{t-1} \expect\|\frac{\bv}{n}\|_2^2\cdot \| \mone^{\T} \prt{\bG^{(m+1)} - \nabla \bF(\bX^{(m+1)})}\|_F^2 \le d^2 n^2 \sigma^2.
    $$
    After summing over $t$ from $0$ to $T$, we get the desired result.
\end{proof}

\begin{lemma}
\label{l.bXQ0}
    For any $T$, we have
    $$
    \begin{aligned}
        & \sum_{t=0}^{T} \expect \norm{\bQ_{0,t,1} + \bQ_{0,t,3} }_F^2 \le 4n^2 d^3 \sigma^2 + 4n^2d^3 \norm{\nabla\bF(\bX^{(0)})}_F^2.
    \end{aligned}
    $$
\end{lemma}
\begin{proof}
    Note that 
    $$
    \norm{\bQ_{0,t,1} + \bQ_{0,t,3} }_F^2 \le 2\norm{\bQ_{0,t,1}}_F^2 + 2\norm{\bQ_{0,t,3}}_F^2.
    $$
    We show the upper bounds for $\sum_{t=0}^{T}\norm{\bQ_{0,t,i}}_F^2$, where $i=1,3$ respectively.
    Based on \cref{l.R0}, \cref{l.sum_matrix} and \cref{l.bound_var}, we have the following result:
    $$
    \begin{aligned}
        & \norm{\bQ_{0,t,1}}_F^2 = \norm{\sum_{j=1}^{\min\{t,d-1\}} \prt{\cR - \frac{1}{n}\mone\bu^{\T}}^j\prt{\cC-\frac{1}{n}\bv\mone^{\T}}^{t-j}\prt{\bG^{(0) }- \nabla\bF(\bX^{(0)})} }_F^2 \\
        &\qquad = \norm{\sum_{j=\max\{1,t-d+1\}}^{\min\{t,d-1\}} \prt{\cR - \frac{1}{n}\mone\bu^{\T}}^j\prt{\cC-\frac{1}{n}\bv\mone^{\T}}^{t-j} \prt{\bG^{(0) }- \nabla\bF(\bX^{(0)})} }_F^2 \\
        &\qquad \le nd \sum_{j=\max\{1,t-d+1\}}^{\min\{t,d-1\}} \norm{\prt{\cC-\frac{1}{n}\bv\mone^{\T}}^{t-j} \prt{\bG^{(0) }- \nabla\bF(\bX^{(0)})}}_F^2.
    \end{aligned}
    $$
    Then, recall that the summation is legal only when $t\le 2(d-2)$. We have
    $$
    \begin{aligned}
        & \sum_{t=0}^{T}\expect \norm{\bQ_{0,t,1}}_F^2 \le \sum_{t=0}^{\min\{T,2(d-2)\}}nd \sum_{j=\max\{1,t-d+1\}}^{\min\{t,d-1\}} \norm{\prt{\cC-\frac{1}{n}\bv\mone^{\T}}^{t-j} \prt{\bG^{(0) }- \nabla\bF(\bX^{(0)})}}_F^2 \\
        & \quad \le 2n^2d^3\sigma^2.
    \end{aligned}
    $$
    Similarly, 
    $$
    \begin{aligned}
        &  \sum_{t=0}^{\T}\expect \norm{\bQ_{0,t,3}}_F^2 \\
        &\qquad \le \sum_{t=0}^{\min\{T,2(d-1)\}}  d \sum_{j=\max\{1,t-d+1\}}^{\min\{t,d-1\}}  \norm{\prt{\cR - \frac{1}{n}\mone\bu^{\T}}^j}_2^2\norm{\prt{\cC-\frac{1}{n}\bv\mone^{\T}}^{t-j} }_2^2 \norm{\nabla \bF(\bX^{(0)})}_F^2\\
        & \qquad \le 2n^2 d^3 \norm{\nabla \bF(\bX^{(0)})}_F^2.
    \end{aligned}
    $$
    Combining the above upper bounds leads to the final result.
\end{proof}

\begin{lemma}
    \label{l.bXQ00}
    For any $T$, we have
    $$
    \sum_{t=0}^{T} \expect \norm{\bQ_{0,t,4} + \bQ_{4,t}}_F^2 \le 2n^2 d^2 L^2 \sum_{t=0}^{T} \expect \|\bPi_{\bu}\bX^{(t)}\|_F^2 + 2n^3 d^2 \sum_{t=0}^{T}\expect  \| f(x_1^{(m)})\|_2^2.
    $$
\end{lemma}
\begin{proof}
    Note that
    $$
    \begin{aligned}
        & \bQ_{0,t,4} + \bQ_{4,t} = \sum_{m=0}^{t-2} \sum_{j=1}^{\min\{t-m-1, d-1\}}\prt{ \cR-\frac{1}{n}\mone\bu^{\T}}^j\frac{\bv\mone^{\T}}{n}\brk{\nabla \bF(\bX^{(m+1)}) - \nabla \bF(\bX^{(m)})} \\
        & \qquad + \sum_{j=1}^{\min\{t,d-1\}}\prt{\cR - \frac{1}{n}\mone\bu^{\T}}^j \frac{1}{n}\bv\mone^{\T}  \nabla \bF(\bX^{(0)}) \\
        & \quad = \sum_{m=\max\{t-d,0\}}^{t-1} \prt{ \cR-\frac{1}{n}\mone\bu^{\T}}^{t-m} \frac{\bv\mone^{\T}}{n} \nabla \bF(\bX^{(m)}),
    \end{aligned}
    $$
    where the last equality comes from extending the summation in the first line and telescoping the summation. Consequently, we have
    $$
    \begin{aligned}
        & \bQ_{0,t,4} + \bQ_{4,t} = \sum_{m=\max\{t-d,0\}}^{t-1} \prt{ \cR-\frac{1}{n}\mone\bu^{\T}}^{t-m} \frac{\bv\mone^{\T}}{n} \prt{\nabla \bF(\bX^{(m)}) - \nabla \bF(\bar{\bX}^{(m)}) + \nabla \bF(\bar{\bX}^{(m)})}.
    \end{aligned}
    $$
    Then, taking the F-norm on both sides and invoking \cref{l.matrix}, \cref{l.sum_matrix} and \cref{l.X_diff} as before, we have
    $$
    \begin{aligned}
        & \norm{\bQ_{0,t,4} + \bQ_{4,t}}_F^2 \le dn \sum_{m=\max\{t-d,0\}}^{t-1} \prt{2\| \frac{\bv\mone^{\T}}{n} \brk{\nabla \bF(\bX^{(m)}) 
        - \nabla \bF(\bar{\bX}^{(m)})}\|_F^2 + 2\| \frac{\bv\mone^{\T}}{n}\nabla \bF(\bar{\bX}^{(m)}) \|_F^2  } \\
        & \le 2d n^2 L^2 \sum_{m=\max\{t-d,0\}}^{t-1} \|\bPi_{\bu}\bX^{(m)}\|_F^2 + 
        2d n^3\sum_{m=\max\{t-d,0\}}^{t-1} \| \nabla f(x_1^{(m)})\|_2^2.
    \end{aligned}
    $$
    Taking expectation on both sides and summing over $t$ from $0$ to $T$, we get
    $$
    \sum_{t=0}^{T} \expect \norm{\bQ_{0,t,4} + \bQ_{4,t}}_F^2 \le 2n^2 d^2 L^2 \sum_{t=0}^{T} \expect \|\bPi_{\bu}\bX^{(t)}\|_F^2 + 2n^3 d^2 \sum_{t=0}^{T}\expect  \| \nabla f(x_1^{(m)})\|_2^2.
    $$
\end{proof}

Back to equation \eqref{eq:ppbX}, note that
$$
\begin{aligned}
    & \norm{\bPi_{\bu}\bX^{(t)}}_F^2 \le  6\norm{(\cR^{\T}-\frac{1}{n}\mone\bu^{\T})^{t}\bX^{(0)}}_F^2 + 6\gamma^2\norm{ \bQ_{0,t,1} + \bQ_{0,t,3} }_F^2 \\
    &\qquad + 6\gamma^2\norm{\bQ_{1,t}}_F^2 + 6\gamma^2\norm{ \bQ_{2,t}}_F^2 + 6\gamma^2\norm{ \bQ_{3,t} + \bQ_{0,t,2}}_F^2 +  6\gamma^2\norm{ \bQ_{0,t,4} + \bQ_{4,t}}_F^2.
\end{aligned}
$$ 
Taking full expectation on both sides, summing over $t$ from $0$ to $T$ and combining \cref{l.bXQ1} to \cref{l.bXQ0}, we have
\begin{equation}
    \begin{aligned}
        & \sum_{t=0}^{T} \expect\norm{\bPi_{\bu}\bX^{(t)}}_F^2 \le 6\sum_{t = 0}^{\min\{T,d-1\}}\expect\norm{\prt{\cR^{\T}-\frac{1}{n}\mone\bu^{\T}}^{t}\bX^{(0)}}_F^2 + 6\gamma^2 \sum_{t=0}^{T}\expect\norm{\bQ_{1,t}}_F^2 + 6\gamma^2\sum_{t=0}^{T}\expect\norm{ \bQ_{2,t}}_F^2 \\
        & \qquad \qquad + 6\gamma^2\sum_{t=0}^{T}\expect\norm{ \bQ_{3,t} + \bQ_{0,t,2}}_F^2 +  6\gamma^2\sum_{t=0}^{T}\expect\norm{ \bQ_{0,t,4} + \bQ_{4,t} }_F^2 + 6\gamma^2\sum_{t=0}^{T} \expect \norm{ \bQ_{0,t,1}  + \bQ_{0,t,3}  }_F^2 \\
        & \quad \le 6nd\norm{\bPi_{\bu}\bX^{(0)}}_F^2 + 192\gamma^2n^2d^4(T+1)\sigma^2  + 312 \gamma^2 n^2 d^4 L^2 \sum_{t=0}^{T} \expect  \norm{\bPi_{\bu}\bX^{(t)}}_F^2\\
        &\qquad\qquad + 1000 \gamma^4 n^4 d^4 (T+1) L^2 \sigma^2+ 2160\gamma^4 n^5 d^4 L^2 \sum_{t=0}^{T}\expect\norm{\nabla f(x_1^{(t)})}_2^2 \\
        & \qquad\qquad + 2160 \gamma^4 n^4 d^5 L^2\norm{\nabla\bF(\bX^{(0)})}_F^2+6\gamma^2 n^2 d^2 (T+1) \sigma^2 + 24\gamma^2 n^2 d^3 \sigma^2\\
        &\qquad\qquad + 24 \gamma^2 n^2 d^3 \norm{\nabla \bF(\bX^{(0)})}_F^2 + 12 \gamma^2 n^2 d^2 L^2 \sum_{t=0}^{T} \expect  \norm{\bPi_{\bu}\bX^{(t)}}_F^2 + 12 \gamma^2n^3 d^2  \sum_{t=0}^{T}\expect\norm{\nabla f(x_1^{(t)})}_2^2.
    \end{aligned}
    \label{eq:ppPiXfinal}
\end{equation}
For $\gamma\le \frac{1}{40nd^2 L}$, which implies that $312 \gamma^2 n^2 d^4 L^2 +  12 \gamma^2 n^2 d^2 L^2 \le \frac{1}{4}$, we can  simplify equation \eqref{eq:ppPiXfinal} as follows:
    $$
    \begin{aligned}
        & \sum_{t=0}^{T} \norm{\bPi_{\bu}\bX^{(t)}}_F^2 
        \le  300 \gamma^2 n^2 d^4 (T+1) \sigma^2 + 20\gamma^2 n^3 d^2 \sum_{t=0}^{T}\expect  \|\nabla f(x_1^{(t)})\|_2^2 \\
        & \qquad \qquad + 6nd\norm{\bPi_{\bu}\bX^{(0)}}_F^2 + 40\gamma^2 n^2 d^3 \norm{\nabla \bF(\bX^{(0)})}_F^2,
    \end{aligned}
    $$
    which implies the desired result.

\end{proof}

\subsubsection{Proof of Lemma \ref{l.inner}}
\label{pf.l.inner}
\begin{proof}
    Notice that
    $$
    x_1^{(t)} = x_1^{(t-1)} - \gamma y_1^{(t-1)} = x_1^{(t-1)} - \gamma \sum_{i\in \cI_{1,1}} y_1^{(t-2)} - \gamma g_1(x_1^{(t-1)},\xi_1^{(t-1)}) + \gamma g_1(x_1^{(t-2)},\xi_1^{(t-2)}).
    $$
    Therefore, $x_1^{(t)}$ does not depend on $\xi_i^{t-1}$ for $i\ne 1$. 
    We iterate the above procedure to get
    $$
    \begin{aligned}
        x_1^{(t)} = & x_1^{(t-1)} - \gamma g_1(x_1^{(t-1)},\xi_1^{(t-1)}) + \gamma g_1(x_1^{(t-2)},\xi_1^{(t-2)})\\
        & - \gamma \sum_{i\in \cI_{1,2}} y_i^{(t-3)} -\gamma \sum_{i\in \cI_{1,1}} g_i(x_1^{(t-2)},\xi_1^{(t-2)}) + \gamma \sum_{i\in \cI_{1,1}} g_i(x_1^{(t-3)},\xi_1^{(t-3)}).
    \end{aligned}
    $$
    Similar to $x_1^{(t)}$, we know that
    $x_1^{(t-1)}$ does not depend on $\xi_i^{(t-2)}$, $i\ne 1$.
    Hence $x_1^{(t)}$ is independent with $\xi_i^{(t-2)}$ for $i\notin\cI_{1,1}$. By iterating the procedure, we conclude that $x_1^{(t)}$ is independent with $\xi_i^{(t-k)}$ for $i\notin\cI_{1,k}$.
    
    Consequently, by choosing $Z = \{\xi_i^{(t-k)}, i \in \cI_{1,k}, i\notin \cI_{1,k-1}\}$, $X = x_1^{(t)}$, $Y = \{ x_i^{(t-k)}, i\in\brk{n}\}$ in \cref{lem:independ_help}, $Z$ is independent with $(X,Y)$, we get
    $$
    \expect\left\langle \nabla f(x_1^{(t)}), \bfe_{\cI_{1,k}/ \cI_{1,k-1}} \prt{\bG^{(t-k)} - \nabla \bF(\bX^{(t-k)})} \right\rangle = 0.
    $$
    Then, invoking \cref{l.Ai}, we have
    $$
    \begin{aligned}
        & \sum_{m=1}^{\min\{t,d\}} \expect \left\langle \nabla f(x_1^{(t)}), \prt{\frac{\bu^{\T}}{n}-\mone^{\T}} \bA_m \prt{\bG^{(t-m)} - \nabla \bF(\bX^{(t-m)})}\right\rangle \\
        = &  \sum_{m=1}^{\min\{t,d\}} \expect\left\langle \nabla f(x_1^{(t)}), \bfe_{\cI_{1,m} / \cI_{1,m-1}} \prt{\bG^{(t-m)} - \nabla \bF(\bX^{(t-m)})} \right\rangle = 0.
    \end{aligned}
    $$
\end{proof}

\subsubsection{Proof of Lemma \ref{l.fbarX}}
\label{pf.l.fbarX}
\begin{proof}
    
By \cref{a.smooth}, the function $f:= \frac{1}{n} \sum_{i=1}^n f_i$ is $L$-smooth. Then,
    \begin{equation}
        \expect f(x_1^{(t+1)})\le \expect f(x_1^{(t)}) + \expect \langle \nabla f(x_1^{(t)}), x_1^{(t+1)} - x_1^{(t)}\rangle + \frac{L}{2} \expect \| x_1^{(t+1)} -x_1^{(t)} \|_2^2.
        \label{eq:fbar}
    \end{equation}
    For the last term, we have
    $$
    \begin{aligned}
        & \expect\|x_1^{(t+1)} - x_1^{(t)}\|_2^2 = \expect\| \frac{1}{n} \bu^{\T} \prt{\bX^{(t+1)}  -\bX^{(t)} } \|_2^2 \\
        = &  \frac{1}{n} \expect \| \frac{1}{n} \mone \bu^{\T} \prt{\bX^{(t+1)}  -\bX^{(t)} } \|_F^2 = \frac{1}{n} \expect \| \Bar{\bX}^{(t+1)} - \Bar{\bX}^{(t)} \|_F^2.
    \end{aligned}
    $$
    For the second last term, we have
    \begin{equation}
    \begin{aligned}
        & \expect \langle \nabla f(x_1^{(t)}), x_1^{(t+1)} - x_1^{(t)}\rangle = \expect \langle \nabla f(x_1^{(t)}), \frac{\bu^{\T}}{n}\prt{\bX^{(t+1)} - \bX^{(t)}}\rangle =  \expect \langle \nabla f(x_1^{(t)}), -\gamma\frac{\bu^{\T}}{n}\bY^{(t)}\rangle \\
    = & \expect \langle \nabla f(x_1^{(t)}), -\gamma\prt{\frac{\bu^{\T}}{n}-\mone^{\T}}\bY^{(t)}\rangle +  \expect \langle \nabla f(x_1^{(t)}), -\gamma\mone^{\T}\bY^{(t)}\rangle.
    \end{aligned} 
    \label{eq:fbarX-1}       
    \end{equation}
    We now bound the two terms in the above equation. Firstly,
    $$
    \begin{aligned}
        & \expect \langle \nabla f(x_1^{(t)}), -\gamma\prt{\frac{\bu^{\T}}{n}-\mone^{\T}}\bY^{(t)}\rangle = 
        \gamma \expect \langle \nabla f(x_1^{(t)}), -\prt{\frac{\bu^{\T}}{n}-\mone^{\T}}\bPi_{\bv}\bY^{(t)}\rangle.
    \end{aligned}
    $$
    Recall that by equation \eqref{eq:pp22},
    \begin{equation}
        \label{eq:lem:inner}
    \begin{aligned}
        &  \gamma\expect\langle \nabla f(x_1^{(t)}), -\prt{\frac{\bu^{\T}}{n}-\mone^{\T}}\bPi_{\bv}\bY^{(t)}\rangle \\
        = & \gamma \expect \left\langle \nabla f(x_1^{(t)}), - \prt{\frac{\bu^{\T}}{n}-\mone^{\T}} \sum_{m=1}^{\min\{t,d\}} \bA_m\bG^{(t-m)} - \prt{\frac{\bu^{\T}}{n}-\mone^{\T}} \bG^{(t)} \right\rangle \\
        = & -\gamma \expect \left\langle \nabla f(x_1^{(t)}), \prt{\frac{\bu^{\T}}{n}-\mone^{\T}} \prt{\bG^{(t)} - \nabla \bF(\bX^{(t)})} \right\rangle \\
        & - \gamma \expect \left\langle \nabla f(x_1^{(t)}), \prt{\frac{\bu^{\T}}{n}-\mone^{\T}} \sum_{m=1}^{\min\{t,d\}} \bA_m \prt{\bG^{(t-m)} - \nabla \bF(\bX^{(t-m)})}\right\rangle \\
        & -\gamma \expect \left\langle \nabla f(x_1^{(t)}), \prt{\frac{\bu^{\T}}{n}-\mone^{\T}} \nabla \bF(\bX^{(t)}) \right\rangle \\
        & -  \gamma \expect \left\langle \nabla f(x_1^{(t)}), \prt{\frac{\bu^{\T}}{n}-\mone^{\T}} \sum_{m=1}^{\min\{t,d\}} \bA_m  \nabla \bF(\bX^{(t-m)}) \right\rangle.
    \end{aligned}
    \end{equation}
    We bound the four terms above one by one. For the first term,
    $$
    \expect\langle \nabla f(x_1^{(t)}), -\prt{\frac{\bu^{\T}}{n}-\mone^{\T}}\prt{\bG^{(t)} - \nabla \bF(\bX^{(t)})}\rangle  =  0.
    $$
    For the second one, invoking \cref{l.inner}, we have
    $$
    \begin{aligned}
        &\expect \left\langle \nabla f(x_1^{(t)}), \prt{\frac{\bu^{\T}}{n}-\mone^{\T}} \sum_{m=1}^{\min\{t,d\}} \bA_m \prt{\bG^{(t-m)} - \nabla \bF(\bX^{(t-m)})}\right\rangle \\
        = & \sum_{m=1}^{\min\{t,d\}} \expect \left\langle \nabla f(x_1^{(t)}), \prt{\frac{\bu^{\T}}{n}-\mone^{\T}} \bA_m \prt{\bG^{(t-m)} - \nabla \bF(\bX^{(t-m)})}\right\rangle = 0.
    \end{aligned}
    $$
    For the last two terms in \eqref{eq:lem:inner}, we have as:
    $$
    \begin{aligned}
        &- \expect \left\langle \nabla f(x_1^{(t)}), \prt{\frac{\bu^{\T}}{n}-\mone^{\T}} \nabla \bF(\bX^{(t)}) \right\rangle -  \expect \left\langle \nabla f(x_1^{(t)}), \prt{\frac{\bu^{\T}}{n}-\mone^{\T}} \sum_{m=1}^{\min\{t,d\}} \bA_m  \nabla \bF(\bX^{(t-m)}) \right\rangle\\
         = & \sum_{m=1}^{\min\{t,d\}}\expect\left\langle  \nabla f(x_1^{(t)}), -\prt{\frac{\bu^{\T}}{n}-\mone^{\T}}\prt{\cC-\frac{1}{n}\bv\mone^{\T}}^{m-1}\prt{\nabla\bF(\bX^{(t-m+1)}) - \nabla \bF(\bX^{(t-m)}) } \right\rangle.
    \end{aligned}
    $$
    By the Cauchy-Schwartz inequality, we have
    $$
    \begin{aligned}
        & \sum_{m=1}^{\min\{t,d\}}\expect\left\langle  \nabla f(x_1^{(t)}), -\prt{\frac{\bu^{\T}}{n}-\mone^{\T}}\prt{\cC-\frac{1}{n}\bv\mone^{\T}}^{m-1}\prt{\nabla\bF(\bX^{(t-m+1)}) - \nabla \bF(\bX^{(t-m)}) } \right\rangle \\
        \le & \sum_{m=1}^{\min\{t,d\}} \crk{\frac{n}{2d}\expect \norm{\nabla f(x_1^{(t)})}_2^2 + \frac{d}{2n}\expect \norm{\prt{\frac{\bu^{\T}}{n}-\mone^{\T}}\prt{\cC-\frac{1}{n}\bv\mone^{\T}}^{m-1}}_2^2\norm{\nabla \bF(\bX^{(t-m+1)}) - \nabla \bF(\bX^{(t-m)})}_F^2}\\
        \le & \frac{n}{2}\expect \norm{\nabla f(x_1^{(t)})}_2^2 + \frac{dL^2}{2}\sum_{m=1}^{\min\{t,d\}} \expect\norm{\bX^{(t-m+1)} - \bX^{(t-m)} }_F^2 \\
        \le & \frac{n}{2}\expect \norm{\nabla f(x_1^{(t)})}_2^2 + \frac{3dL^2}{2}\sum_{m=1}^{\min\{t,d\}} \expect\prt{\norm{\bPi_{\bu}\bX^{(t-m+1)}}_F^2 + \norm{\bPi_{\bu}\bX^{(t-m)}}_F^2 + \norm{\bar{\bX}^{(t-m+1)} - \bar{\bX}^{(t-m)}}_F^2}.
    \end{aligned}
    $$
    Thus, combining the above inequalities together yields
    $$
    \begin{aligned}
        &\sum_{t=0}^{T} \gamma\expect\langle \nabla f(x_1^{(t)}), -\prt{\frac{\bu^{\T}}{n}-\mone^{\T}}\bPi_{\bv}\bY^{(t)}\rangle \le  \frac{n\gamma}{2} \sum_{t=0}^{T} \expect \norm{\nabla f(x_1^{(t)})}_2^2 \\
        &\qquad \qquad+ 3\gamma d^2 L^2 \sum_{t=0}^{T} \expect\norm{\bPi_{\bu}\bX^{(t)}}_F^2 + 3\gamma d^2L^2 \sum_{t=0}^{T} \expect\norm{\bar{\bX}^{(t+1)} - \bar{\bX}^{(t)}}_F^2.
    \end{aligned}
    $$
    Secondly, for the second term in \eqref{eq:fbarX-1},
    $$
    \begin{aligned}
        & \expect \langle \nabla f(x_1^{(t)}), -\gamma\mone^{\T}\bY^{(t)}\rangle = \expect \langle \nabla f(x_1^{(t)}), -\gamma\mone^{\T}\bG^{(t)}\rangle = \expect \langle \nabla f(x_1^{(t)}), -\gamma\mone^{\T} \nabla \bF(\bX^{(t)})\rangle \\
        = & -n\gamma \expect \| \nabla f(x_1^{(t)}) \|_2^2 - n\gamma \expect \langle \nabla f(x_1^{(t)}), \frac{1}{n} \mone^{\T} \nabla \bF(\bX^{(t)}) - \frac{1}{n}\mone^{\T} \nabla \bF(\Bar{\bX}^{(t)})\rangle \\
        \le & -n\gamma \expect \| \nabla f(x_1^{(t)}) \|_2^2 + \gamma \frac{n}{4} \expect \| \nabla f(x_1^{(t)}) \|_2^2 + 2\gamma  \expect \| \nabla \bF(\bX^{(t)}) - \nabla \bF(\Bar{\bX}^{(t)})\|_F^2 \\
        \le &  -n\gamma \expect \| \nabla f(x_1^{(t)}) \|_2^2 + \gamma \frac{n}{4} \expect \| \nabla f(x_1^{(t)}) \|_2^2 + 2\gamma L^2 \expect \| \bPi_{\bu}\bX^{(t)}\|_F^2.
    \end{aligned}
    $$
    Summing over $t$ from $0$ to $T$, we have
    $$
    \sum_{t=0}^{T} \expect \langle \nabla f(x_1^{(t)}), -\gamma\mone^{\T}\bY^{(t)}\rangle \le -\frac{3n \gamma}{4}\sum_{t=0}^{T}\expect\norm{\nabla f(x_1^{(t)})}_2^2 + 2\gamma L^2 \sum_{t=0}^{T}\expect\norm{\bPi_{\bu}\bX^{(t)}}_F^2.
    $$
    It yields by summing over $t$ from 0 to $T$ on both sides of equation \eqref{eq:fbar} that
    $$
    \begin{aligned}
       & \expect f(x_1^{(T+1)}) - \expect f(x_1^0) \le \sum_{t=0}^{T}\expect \langle \nabla f(x_1^{(t)}), x_1^{(t+1)} - x_1^{(t)}\rangle +\frac{L}{2n}\sum_{t=0}^{T} \expect \| \Bar{\bX}^{(t+1)} - \Bar{\bX}^{(t)} \|_F^2\\
       & \quad \le  \sum_{t=0}^{T} \expect \langle \nabla f(x_1^{(t)}), -\gamma\mone^{\T}\bY^{(t)}\rangle + \sum_{t=0}^{T} \gamma\expect\langle \nabla f(x_1^{(t)}), -\prt{\frac{\bu^{\T}}{n}-\mone^{\T}}\bPi_{\bv}\bY^{(t)}\rangle +\frac{L}{2n}\sum_{t=0}^{T} \expect \| \Bar{\bX}^{(t+1)} - \Bar{\bX}^{(t)} \|_F^2.
    \end{aligned}
    $$
    With the results above, given $\Delta_f = f(x^{0}) - f^*$, we have
    $$
    \begin{aligned}
        -\Delta_f &\le -\frac{n \gamma}{4}\sum_{t=0}^{T}\expect\norm{\nabla f(x_1^{(t)})}_2^2 + \frac{ L}{2n}\sum_{t=0}^{T} \expect \| \Bar{\bX}^{(t+1)} - \Bar{\bX}^{(t)} \|_F^2\\
        & \qquad  + 5\gamma d^2 L^2 \sum_{t=0}^{T} \expect\norm{\bPi_{\bu}\bX^{(t)}}_F^2 + 3\gamma d^2L^2 \sum_{t=0}^{T} \expect\norm{\bar{\bX}^{(t+1)} - \bar{\bX}^{(t)}}_F^2.
    \end{aligned}
    $$
    Invoking \cref{l.X_diff}, we have for $\gamma\le \frac{1}{100nd^3 L} (\le \frac{1}{10ndL})$ that
    $$
    \begin{aligned}
        -\Delta_f &\le n \gamma\prt{45\gamma^2 n^2 d^2L^2 +8\gamma n  L -\frac{1}{4}}\sum_{t=0}^{T}\expect\norm{\nabla f(x_1^{(t)})}_2^2 \\
        & \qquad  + \prt{150\gamma^3 n^2 d^4 L^4+25\gamma^2 n d^2 L^3 + 5\gamma d^2 L^2}\sum_{t=0}^{T} \expect\norm{\bPi_{\bu}\bX^{(t)}}_F^2 \\
        & \qquad + 3\gamma^2 n (T+1) L \sigma^2 + 18\gamma^3 n^2 d^2 L^2 (T+1) \sigma^2 \\
        & \qquad + \prt{18\gamma^3 n^2 d^3L^2 +3\gamma^2 n L} \norm{\nabla \bF(\bX^{(0)})}_F^2,
    \end{aligned}
    $$
    where it holds that $150\gamma^3 n^2 d^4 L^4+25\gamma^2 n d^2 L^3 + 5\gamma d^2 L^2 \le 6\gamma d^2L^2$.
    
    Invoking \cref{l.PiX}, we have for $\gamma\le \frac{1}{100nd^3 L} (\le \frac{1}{40nd^2 L})$ that
    $$
    \begin{aligned}
        -\Delta_f &\le n \gamma\prt{120\gamma^2n^2d^4 L^2+45\gamma^2 n^2 d^2L^2 +8\gamma n L -\frac{1}{4}}\sum_{t=0}^{T}\expect\norm{\nabla f(x_1^{(t)})}_2^2 \\
        & \qquad + 3\gamma^2 n (T+1) L \sigma^2 + 18\gamma^3 n^2 d^2 L^2 (T+1) \sigma^2 + 1800 \gamma^3 n^2 d^6(T+1)\sigma^2 L^2 \\
        & \qquad + \prt{240\gamma^3 n^2 d^5 L^2 + 18\gamma^3 n^2 d^3L^2 +3\gamma^2 n L} \norm{\nabla \bF(\bX^{(0)})}_F^2 \\
        & \qquad  + 36\gamma nd^3 L^2 \norm{\bPi_{\bu}\bX^{(0)}}_F^2.
    \end{aligned}
    $$
    Thus, for $\gamma\le \frac{1}{100n d^3 L}$, we have
    $$
    \begin{aligned}
    & 120\gamma^2n^2d^4L^2+45\gamma^2 n^2 d^2L^2 +8\gamma n L -\frac{1}{4} \le - \frac{1}{8}        \\
    & \gamma n \prt{240\gamma^2 n d^5 L^2 + 18\gamma^2 n d^3L^2 +3\gamma L } \le 7\gamma^2 n d^2 L.
    \end{aligned}
    $$ 
    After re-arranging the terms, we conclude that
    $$
    \begin{aligned}
        &\frac{1}{T+1} \sum_{t=0}^{T}\expect\norm{\nabla f(x_1^{(t)})}_2^2 \le \frac{8\Delta_f}{\gamma n (T+1)} + 24\gamma  \sigma^2 L+ 20000\gamma^2 n d^6 \sigma^2 L^2 \\
        &\qquad + \frac{400 d^3 L^2\norm{\bPi_{\bu}\bX^{(0)}}_F^2 }{T+1} + \frac{56 \gamma d^3 L\norm{\nabla \bF(\bX^{(0)})}_F^2}{T+1}.
    \end{aligned}
    $$
\end{proof}

\subsubsection{Proof of Lemma \ref{l.x-star}}
\label{pf.l.x-star}
 Let $x^* = \arg\min_x f(x)$. We start with analyzing the behavior of $\|x_1^{(t)} -  x^*\|^2$ after obtaining \cref{l.inner}. It holds that
    \begin{equation}
    \label{eq:x-star1}
        \left\|x_1^{(t+1)} - x^* \right\|^2 = \left\|x_1^{(t)} - x^* \right\|^2 + 2 \left\langle x_1^{(t)} - x^*, x_1^{(t+1)} - x_1^{(t)} \right\rangle + \left\|x_1^{(t+1)} - x_1^{(t)} \right\|^2.
    \end{equation}
    To deal with the critical inner product, similar to the decomposition in Equation \eqref{eq:fbarX-1}, we have, by replacing $\nabla f(x_1^{(t)})$ with $x_1^{(t)} - x^*$ in Equation \eqref{eq:lem:inner}, and invoking \cref{l.inner} as we have done in \cref{l.fbarX}, that
    \begin{equation}
        \begin{aligned}
            & \left\langle x_1^{(t)} - x^*, x_1^{(t+1)} - x_1^{(t)} \right\rangle \\
            = & -\left\langle x_1^{(t)} - x^*,\gamma \left(\frac{\mathbf{u}^\top}{n} - \mathbf{1}^\top \right)\mathbf{\Pi}_{\mathbf{v}}\mathbf{Y}^{(t)} \right\rangle - \left\langle x_1^{(t)} - x^*,  \gamma \mathbf{1}^\top \mathbf{Y}^{(t)} \right\rangle.
        \end{aligned}
    \end{equation}
    and
    \begin{equation}
        \begin{aligned}
            &-\mathbb{E}\left\langle x_1^{(t)} - x^*,\gamma \left(\frac{\mathbf{u}^\top}{n} - \mathbf{1}^\top \right)\mathbf{\Pi}_{\mathbf{v}}\mathbf{Y}^{(t)} \right\rangle\\
            = & - \gamma \mathbb{E} \left\langle x_1^{(t)} - x^*,\left(\frac{\mathbf{u}^\top}{n} - \mathbf{1}^\top \right) \nabla \mathbf{F}(\mathbf{X}^{(t)}) \right\rangle\\
            & - \gamma \mathbb{E} \left\langle x_1^{(t)} - x^*,\left(\frac{\mathbf{u}^\top}{n} - \mathbf{1}^\top \right)\sum_{m=1}^{\min\{t,d\} } \mathbf{A}_m\nabla \mathbf{F}(\mathbf{X}^{(t)}) \right\rangle \\
            \le & \frac{n\gamma \mu}{4} \mathbb{E} \left\|x_1^{(t)} - x^* \right\|^2 
            + \frac{3dL^2}{\mu} \gamma \sum_{m=1}^{ \min\{t,d\} } \mathbb{E} \left( \left\|\mathbf{\Pi}_{\mathbf{u}} \mathbf{X}^{t-m+1}  \right\|^2 + \left\|\mathbf{\Pi}_{\mathbf{u}} \mathbf{X}^{t-m}  \right\|^2 + \left\|\bar{\mathbf{X}}^{t-m+1} - \bar{\mathbf{X}}^{t-m} \right\|^2\right).
        \end{aligned}
    \end{equation}
    Notice that, first by strong convexity of $f$ (\cref{a.convex}) and then by $L$-smoothness (\cref{a.smooth}), there holds 
    \begin{equation}
    \begin{aligned}
        &-  \left\langle x_1^{(t)} - x^*,  \nabla f(x_1^{(t)}) \right\rangle \le f(x^*)-f(x_1^{(t)})  - \frac{\mu}{2} \left\|x_1^{(t)} - x^* \right\|^2 \\
        \le & -\frac{1}{2L} \left\|\nabla f(x_1^{(t)} ) \right\|^2  - \frac{\mu}{2} \left\|x_1^{(t)} - x^* \right\|^2.
    \end{aligned}
    \end{equation}
    Then, 
    \begin{equation}
        \begin{aligned}
            & - \gamma \mathbb{E} \left\langle x_1^{(t)} - x^*,  \mathbf{1}^\top \mathbf{Y}^{(t)} \right\rangle = - \gamma \mathbb{E} \left\langle x_1^{(t)} - x^*,  \mathbf{1}^\top \nabla \mathbf{F}(\mathbf{X}^{(t)}) \right\rangle \\
            = & - n \gamma \mathbb{E} \left\langle x_1^{(t)} - x^*,  \nabla f(x_1^{(t)}) \right\rangle - \gamma \mathbb{E} \left\langle x_1^{(t)} - x^*,  \mathbf{1}^\top \nabla \mathbf{F}(\mathbf{X}^{(t)}) - \mathbf{1}^\top \nabla \mathbf{F}(\bar{\mathbf{X}}^{(t)}) \right\rangle \\
            \le & -\frac{n \gamma \mu}{2}\mathbb{E}\left\|x_1^{(t)} - x^* \right\|^2-\frac{n\gamma}{2L} \left\|\nabla f(x_1^{(t)} ) \right\|^2 + \frac{n\gamma \mu}{8} \mathbb{E}\left\|x_1^{(t)} - x^* \right\|^2 + \frac{2L^2\gamma}{\mu}\mathbb{E} \left\|\mathbf{\Pi}_{\mathbf{u}} \mathbf{X}^{(t)}  \right\|^2.
        \end{aligned}
    \end{equation}
    Thus, plugging the above results into Equation \eqref{eq:x-star1}, with $\kappa:=L/\mu$ as the conditional number, it holds that
    \begin{equation}
    \begin{aligned}
    & \mathbb{E}\left\|x_1^{(t+1)} - x^* \right\|^2 \le \left(1 - \frac{n\gamma \mu}{4} \right)\mathbb{E}\left\|x_1^{(t)} - x^* \right\|^2  -\frac{n\gamma}{L} \mathbb{E}\left\|\nabla f(x_1^{(t)} ) \right\|^2 + \frac{1}{n}\mathbb{E}\left\|\bar{\mathbf{X}}^{(t+1)} - \bar{\mathbf{X}}^{(t)} \right\|^2_F \\
    &\quad + 6 d\kappa L \gamma \sum_{m=1}^{ \min\{t,d\} } \mathbb{E}\left\|\bar{\mathbf{X}}^{(t-m+1)} - \bar{\mathbf{X}}^{(t-m)} \right\|^2_F + 20d\kappa L \gamma \sum_{m=0}^{ \min\{t,d\}+1 } \mathbb{E} \left\|\mathbf{\Pi}_{\mathbf{u}} \mathbf{X}^{t-m}  \right\|^2_F.
    \end{aligned}
    \end{equation}
    We derive the convergence result by several standard steps as follows.\\
    \textbf{Step 1}. Unwinding the above recursion, it follows by $\frac{1}{2}\le 1 - \frac{n\gamma \mu}{4} \le 1$ for $\gamma \le \frac{1}{10 nd^2\kappa L}$ that
\begin{equation}
    \begin{aligned}
    &\mathbb{E}\left\|x_1^{(T)} - x^* \right\|^2 \le \left(1 - \frac{n\gamma \mu}{4} \right)^{T}\left\|x_1^{(0)} - x^* \right\|^2  -\frac{n\gamma}{2L} \sum_{t=0}^{T}\mathbb{E}\left\|\nabla f(x_1^{(t)} ) \right\|^2 \\
    &\quad + \left(\frac{1}{n} + 6d^2\kappa L \gamma \right)\sum_{t=0}^{T}\prt{1-\frac{n\gamma\mu}{4}}^{T-t}\mathbb{E}\left\|\bar{\mathbf{X}}^{(t+1)} - \bar{\mathbf{X}}^{(t)} \right\|^2_F \\
    & \quad + 60d^2\kappa L \gamma \sum_{t=0}^{T}\prt{1-\frac{n\gamma\mu}{4}}^{T-t} \mathbb{E} \left\|\mathbf{\Pi}_{\mathbf{u}} \mathbf{X}^{t}  \right\|^2_F.
    \end{aligned}
\end{equation}
\textbf{Step 2}. To refine \cref{l.X_diff}, we start with Equation \eqref{eq:refine-X-diff} multiplied by the coefficient $\prt{1-\frac{n\gamma\mu}{4}}^{T-t}$ before summing over $t$ in Equation \eqref{eq:diff_X_1} from $0$ to $T$. Then, we have, for $\gamma \le \frac{1}{10 nd^2\kappa L}$ ($\frac{1}{2}\le 1 - \frac{n\gamma \mu}{4} \le 1$ for $\gamma \le \frac{1}{10 nd^2\kappa L}$), that 
\begin{equation}
    \label{eq:step2-refine}
    \begin{aligned}
        & \sum_{t=0}^{T}  \prt{1-\frac{n\gamma\mu}{4}}^{T-t} \expect\|\bar{\bX}^{(t+1)} - \Bar{\bX}^{(t)}\|_F^2 \le  6\gamma^2 n^2\sigma^2(T+1) + 50\gamma^2n^2d^2L^2 \sum_{t=0}^{T}\prt{1-\frac{n\gamma\mu}{4}}^{T-t}\expect \norm{\bPi_{\bu}\bX^{(t)}}_F^2 \\
        &\qquad + 6\gamma^2 n^2 \sum_{t=0}^{d}\prt{1-\frac{n\gamma\mu}{4}}^{T-t} \norm{\nabla\bF(\bX^{(0)})}_F^2  + 15\gamma^2 n^3\sum_{t=0}^{T}\expect\norm{\nabla f(x_1^{(t)})}_2^2.
    \end{aligned}
\end{equation}
where we get the result which only modifies the coefficient of the term $\norm{\nabla \bF(\bX^{(0)})}_F^2$.

Implementing the above result, we have, for $\gamma\le \frac{1}{10 nd^2\kappa L} \prt{\le \frac{1}{10 n d L}} $,
$$
\begin{aligned}
    &\mathbb{E}\left\|x_1^{(T)} - x^* \right\|^2 \le \left(1 - \frac{n\gamma \mu}{4} \right)^{T}\left\|x_1^{(0)} - x^* \right\|^2 \\
    & \quad + \prt{-\frac{n\gamma}{2L} + 15\gamma^2 n^2   + 90\gamma^3 n^3 d^2 \kappa L}\sum_{t=0}^{T}\mathbb{E}\left\|\nabla f(x_1^{(t)} ) \right\|^2 \\
    & \quad + 6\gamma^2 n \sigma^2 \prt{T+1} + 36\gamma^3 n^2 d^2 \kappa L \sigma^2 \prt{T+1}\\
    & \quad + \prt{60\gamma d^2\kappa L  + 50\gamma^2 n d^2 L^2 + 300 \gamma^3 n^2 d^4\kappa L^3 }\sum_{t=0}^{T}\prt{1-\frac{n\gamma\mu}{4}}^{T-t}\mathbb{E} \left\|\mathbf{\Pi}_{\mathbf{u}} \mathbf{X}^{t}  \right\|^2_F\\
    &\quad + \prt{6\gamma^2 n d + 36 \gamma^3 n^2 d^3 \kappa L}\prt{1-\frac{n\gamma\mu}{4}}^{T-d}\norm{\nabla \bF(\bX^{(0)})}.
\end{aligned}
$$
\textbf{Step 3}. Similarly, we refine \cref{l.PiX} as follows. By multiplying $\prt{1-\frac{n\gamma\mu}{4}}^{T-t}$ before summing over $t$ in Equation \eqref{eq:ppPiXfinal}, we have
$$
    \begin{aligned}
        & \sum_{t=0}^{T} \prt{1-\frac{n\gamma\mu}{4}}^{T-t}\norm{\bPi_{\bu}\bX^{(t)}}_F^2 
        \le  300 \gamma^2 n^2 d^4 (T+1) \sigma^2 + 20\gamma^2 n^3 d^2 \sum_{t=0}^{T}\expect  \|\nabla f(x_1^{(t)})\|_2^2 \\
        & \qquad \qquad + 6nd\norm{\bPi_{\bu}\bX^{(0)}}_F^2 + 40\gamma^2 n^2 d^3\prt{1-\frac{n\gamma\mu}{4}}^{T-d} \norm{\nabla \bF(\bX^{(0)})}_F^2,
    \end{aligned}
$$

Implementing the above result, we have for $\gamma \le \frac{1}{40 n d^2 \kappa L}$ that
$$
60\gamma d^2\kappa L  + 50\gamma^2 n d^2 L^2 + 300 \gamma^3 n^2 d^4\kappa L^3 \le 70 \gamma d^2\kappa L,
$$
and
$$
\begin{aligned}
     &\mathbb{E}\left\|x_1^{(T)} - x^* \right\|^2 \le \left(1 - \frac{n\gamma \mu}{4} \right)^{T}\left\|x_1^{(0)} - x^* \right\|^2 \\
     & \quad + \prt{-\frac{n\gamma}{2L} + 15\gamma^2 n^2   + 90\gamma^3 n^3 d^2 \kappa L + 1400\gamma^3 n^3 d^4 \kappa L }\sum_{t=0}^{T}\mathbb{E}\left\|\nabla f(x_1^{(t)} ) \right\|^2 \\
     & \quad + 6\gamma^2 n \sigma^2 \prt{T+1} + 36\gamma^3 n^2 d^2 \kappa L \sigma^2 \prt{T+1} + 21000 \gamma^3 n^2  d^6 \kappa L\sigma^2\prt{T+1}\\
     &\quad + \prt{6\gamma^2 n d + 36 \gamma^3 n^2 d^3 \kappa L + 2800\gamma^3n^2 d^5\kappa L}\prt{1-\frac{n\gamma\mu}{4}}^{T-d}\norm{\nabla \bF(\bX^{(0)})} + 420\gamma^3 n^2 d^3 \kappa L \norm{\bPi_{\bu}\bX^{(0)}}_F^2 .
\end{aligned}
$$
\textbf{Step 4}. Note that the coefficient of the term $ \sum_{t=0}^{T}\left\|\nabla f(x_1^{(t)} ) \right\|^2$ is smaller than $0$ for $\gamma \le \frac{1}{100 n d^2 \kappa L}$, derived as follows:
$$
\begin{aligned}
    &-\frac{n\gamma}{2L} + 15\gamma^2 n^2   + 90\gamma^3 n^3 d^2 \kappa L + 1400\gamma^3 n^3 d^4 \kappa L \\
    = & \frac{n\gamma}{2L} \prt{-1 + 30 \gamma n L + 180\gamma^2 n^2 d^2 \kappa L^2 + 2800\gamma^2 n d^4 \kappa L} \\
    \le & \frac{n\gamma}{2L} \prt{-1 + \frac{30}{100} +\frac{180}{10^4} + \frac{2800}{10^4} } \le -\frac{\gamma n}{20L} < 0.
\end{aligned}
$$
Thus, it holds that
$$
\begin{aligned}
     &\mathbb{E}\left\|x_1^{(T)} - x^* \right\|^2 \le \left(1 - \frac{n\gamma \mu}{4} \right)^{T}\left\|x_1^{(0)} - x^* \right\|^2 \\
     & \quad + 7\gamma^2 n \sigma^2 \prt{T+1} + 21000 \gamma^3 n^2  d^6 \kappa L\sigma^2\prt{T+1}\\
     &\quad + 80\gamma^2 n d^3\prt{1-\frac{n\gamma\mu}{4}}^{T-d}\norm{\nabla \bF(\bX^{(0)})} + 420\gamma^3 n^2 d^3 \kappa L \norm{\bPi_{\bu}\bX^{(0)}}_F^2 .
\end{aligned}
$$
\subsection{Proof of the Convergence Results}

\subsubsection{Proof of Theorem \ref{thm:convergence}}
    Invoking \cref{l.fbarX}, with identical initial values $x_i^{(0)}$ that implies $\norm{\bPi_{\bu}\bX^{(0)}}_F^2 = 0$, we have
    $$
    \begin{aligned}
        &\frac{1}{T+1} \sum_{t=0}^{T}\expect\norm{\nabla f(x_1^{(t)})}_2^2 \le \frac{8\Delta_f}{\gamma n (T+1)} + 24\gamma  \sigma^2 L+ 2\cdot 10^4\gamma^2 n d^6 \sigma^2 L^2 + \frac{56 \gamma d^3 L\norm{\nabla \bF(\bX^{(0)})}_F^2}{T+1}.
    \end{aligned}
    $$   
    Referring to Lemma 26 in \cite{koloskova2021improved}, as stated in \cref{l.martin}, by taking $A = \frac{8\Delta_f}{n}$, $B = 24\sigma^2 L$, $C = 20000nd^6 \sigma^2 L^2$ and $\alpha = 100nd^3L$,  when considering $\gamma = \min\{\prt{\frac{\Delta_f}{3\sigma^2 L n (T+1)}}^{\frac{1}{2}}, \prt{\frac{\Delta_f}{1500n^2 d^6 \sigma^2 L^2 (T+1)}}^{\frac{1}{3}}, \frac{1}{100nd^3L}\}$, we have
    $$
    \begin{aligned}
        & \frac{1}{T+1} \sum_{t=0}^{T}\expect\norm{\nabla f(x_1^{(t)})}_2^2 
        \le \frac{32 \sqrt{\Delta_f \sigma^2 L}}{\sqrt{n(T+1)}} + \frac{240d^2 \prt{\sigma^2 L^2 \Delta_f^2}^{\frac{1}{3}}}{\prt{\sqrt{n} (T+1)}^{\frac{2}{3}}   } + \frac{800d^3 L \Delta_f }{T+1}  + \frac{\norm{\nabla \bF(\bX^{(0)})}_F^2}{n(T+1)}.
    \end{aligned}
    $$
    

\subsubsection{Proof of Theorem \ref{thm:convergence-convex}}
 Invoking \cref{l.x-star}, with identical values $x_i^{(0)}$ that implies $\norm{\bPi_{\bu}\bX^{(0)}}_F^2=0$, we have 
$$
\begin{aligned}
     &\mathbb{E}\left\|x_1^{(T)} - x^* \right\|^2 \le \left(1 - \frac{n\gamma \mu}{4} \right)^{T}\left\|x_1^{(0)} - x^* \right\|^2 \\
     & \quad + 7\gamma^2 n \sigma^2 \prt{T+1} + 21000 \gamma^3 n^2  d^6 \kappa L\sigma^2\prt{T+1}\\
     &\quad + 80\gamma^2 n d^3\prt{1-\frac{n\gamma\mu}{4}}^{T-d}\norm{\nabla \bF(\bX^{(0)})} .
\end{aligned}
$$
Considering $\gamma = \min\crk{\frac{1}{100nd^2\kappa L}, \frac{16\log(n\prt{T+1}^2)}{n\prt{T+1}\mu}}$, for $T\ge 2d$ we get that
$$
\begin{aligned}
\prt{1-\frac{n\mu\gamma}{4}}^T \le &  \max\crk{\prt{1-\frac{1}{400d^2\kappa^2}}^T, \prt{1-\frac{4\log(n\prt{T+1}^2)}{T+1}}^T}\\
\le & \max\crk{\exp(-\frac{T}{400d^2\kappa^2}), \frac{40}{n\prt{T+1}^2}},
\end{aligned}
$$
and
$$
\begin{aligned}
\prt{1-\frac{n\mu\gamma}{4}}^{T-d} \le &  \max\crk{\prt{1-\frac{1}{400d^2\kappa^2}}^{T/2}, \prt{1-\frac{4\log(n\prt{T+1}^2)}{T+1}}^{T/2}}\\
\le & \max\crk{\exp(-\frac{T}{800d^2\kappa^2}), \frac{10}{n\prt{T+1}^2}},
\end{aligned}
$$
where we use the fact $(1-\frac{1}{x})^{x} \le e^{-1}$ and $1-x\le \exp(-x)$ for any $x\in\reals_{+}$. Then,
$$
\begin{aligned}
    &\mathbb{E}\left\|x_1^{(T)} - x^* \right\|^2 \le  \frac{2240\sigma^2 \log(n(T+1)^2)}{n\prt{T+1}\mu^2} + \frac{26880000 d^6 \kappa^2 \sigma^2\prt{\log(n(T+1)^2)}^2}{n\prt{T+1}^2\mu^2} \\
    & \quad + \max\crk{\exp(-\frac{T}{800d^2\kappa^2}), \frac{40}{n\prt{T+1}^2}}\prt{\left\|x_1^{(0)} - x^* \right\|^2 + \frac{1}{nL^2}\norm{\nabla\bF(\bX^{(0)})}_F^2  }.
\end{aligned}
$$

\section{Additional Experiments}
\label{appendix:exp}
For the problem of training a CNN on the MNIST dataset, we have further compared the real-time performance of BTPP with other representative methods. The experiments are conducted on a server equipped with eight Nvidia RTX 3090 GPUs and two Intel Xeon Gold 4310 CPUs, where the communication between GPUs follows the topology requirement of each algorithm. We measure the running time including GPU computation and communication for 13,000 iterations. The experimental settings are consistent with those described in \cref{sec:deep}.
\begin{figure}[ht]
    \centering
    \begin{minipage}{0.49\textwidth}
        \centering
        \includegraphics[width=\linewidth]{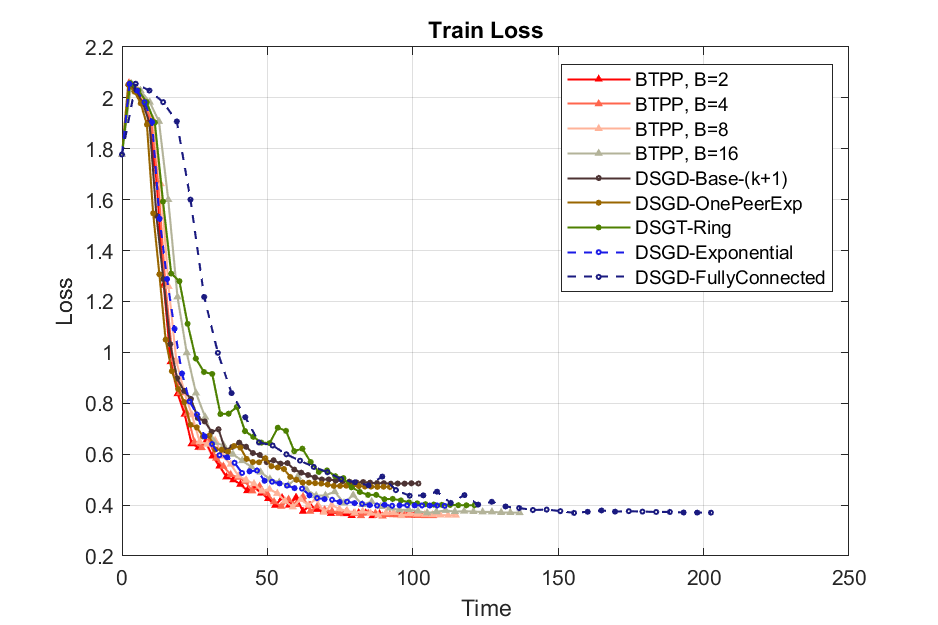} 
    \end{minipage}
    \hfill 
    \begin{minipage}{0.49\textwidth}
        \centering
        \includegraphics[width=\linewidth]{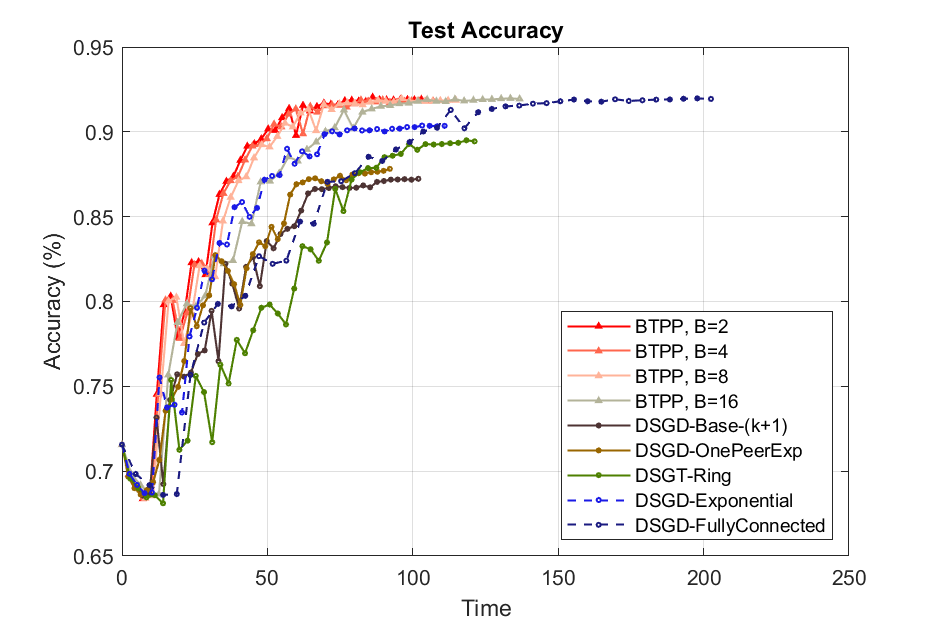} 
    \end{minipage}
    \caption{Real-time performance of BTPP (with different branch size $B$) compared with related methods when training CNN over MNIST.}
    \label{fig:mnist-real-time}
\end{figure}
 From \cref{fig:mnist-real-time}, BTPP outperforms the other algorithms concerning the running time. Additionally, we evaluate BTPP with various branch sizes $B$, concluding that for relatively small values of $n$, a branch size of $B=2$ is most effective.
 
 Furthermore, we consider training VGG13 on the CIFAR10 dataset, with $n=8$ and a batch size of 16. The learning rate and topology configurations are consistent with those described in \cref{sec:deep}. Additionally, the case of BTPP with $B=8$ is equivalent to DSGD in a fully connected setting, meaning that they produce identical outputs when using the same random seed. \cref{fig:vgg13-iteration} and \cref{fig:vgg13-real-time} illustrate that BTPP beats competing algorithms in terms of the convergence rate (against iteration number) and running time. Moreover, a branch size of $B=2$ is optimal.
 \begin{figure}[ht]
    \centering
    \begin{minipage}{0.49\textwidth}
        \centering
        \includegraphics[width=\linewidth]{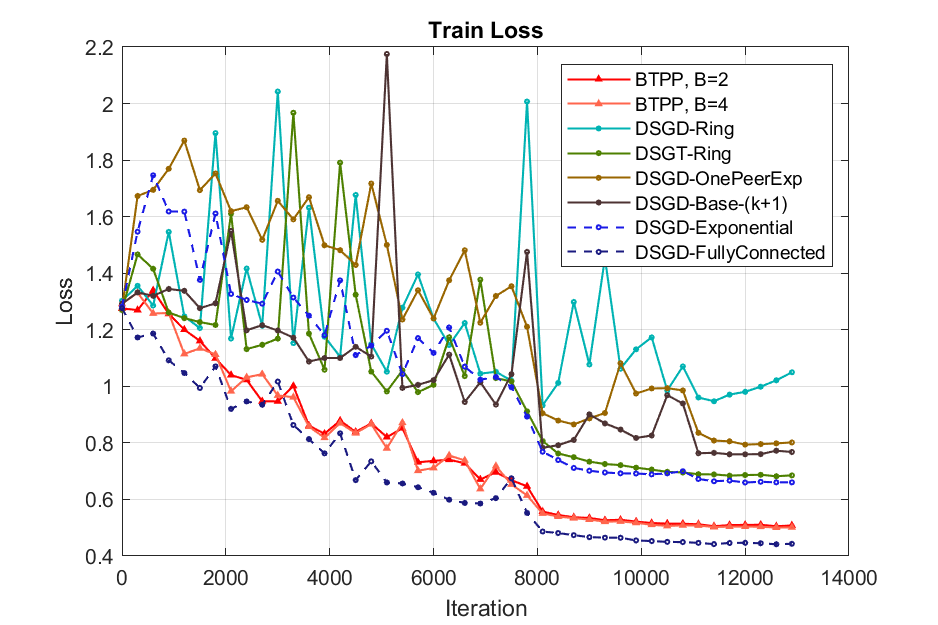} 
    \end{minipage}
    \hfill 
    \begin{minipage}{0.49\textwidth}
        \centering
        \includegraphics[width=\linewidth]{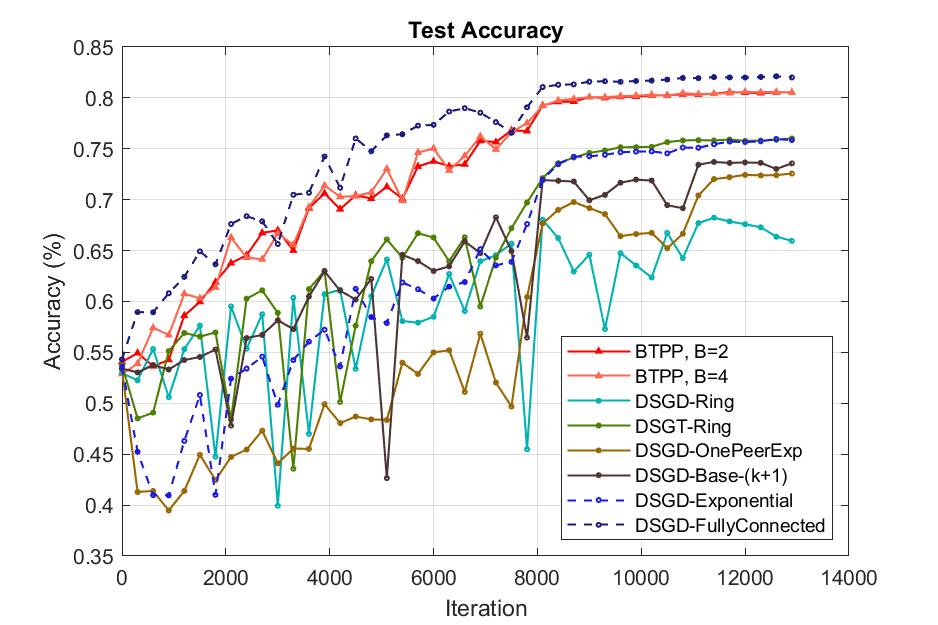} 
    \end{minipage}
    \caption{Performance of BTPP (with different branch size B) compared with related methods for training VGG13 over CIFAR10.}
    \label{fig:vgg13-iteration}
\end{figure}
\begin{figure}[ht]
    \centering
    \begin{minipage}{0.49\textwidth}
        \centering
        \includegraphics[width=\linewidth]{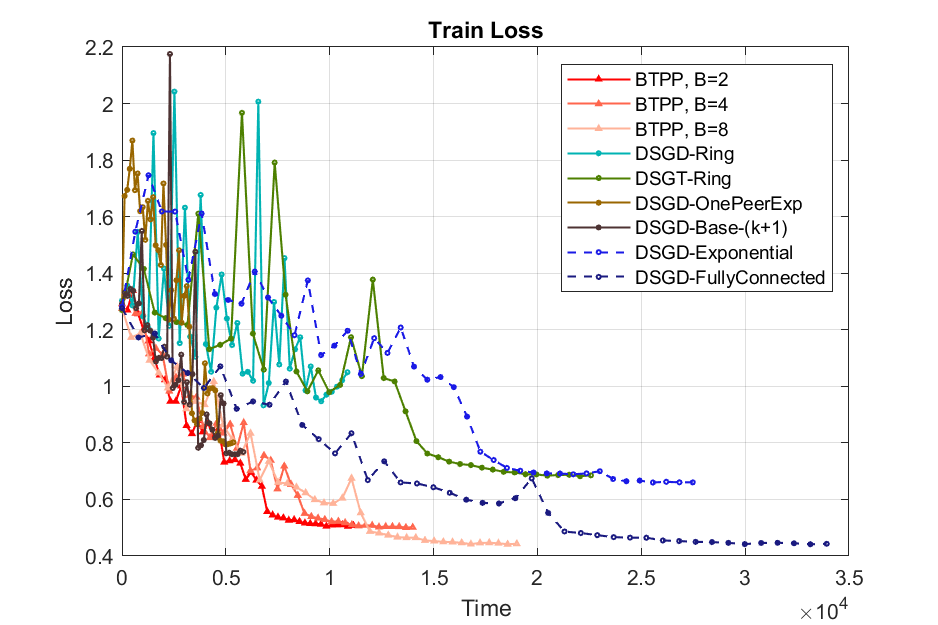} 
    \end{minipage}
    \hfill 
    \begin{minipage}{0.49\textwidth}
        \centering
        \includegraphics[width=\linewidth]{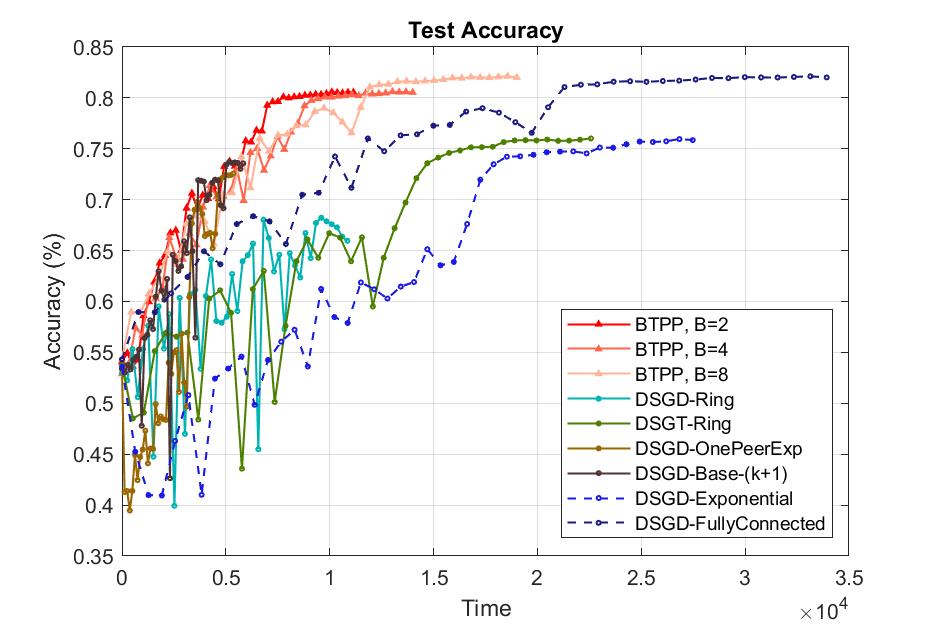} 
    \end{minipage}
    \caption{Real-time performance of BTPP (with different branch size $B$) compared with related methods when training VGG13 over CIFAR10.}
    \label{fig:vgg13-real-time}
\end{figure}

We further demonstrate that the performance of BTPP can be improved by incorporating a momentum term (with momentum parameter set to 0.9) when data heterogeneity exists or by removing the data heterogeneity, which involves randomly assigning samples to each agent; see \cref{fig:cnn_compare} and \cref{fig:vgg_compare}.

 \begin{figure}[!htbp]
    \centering
    \begin{minipage}{0.49\textwidth}
        \centering
        \includegraphics[width=\linewidth]{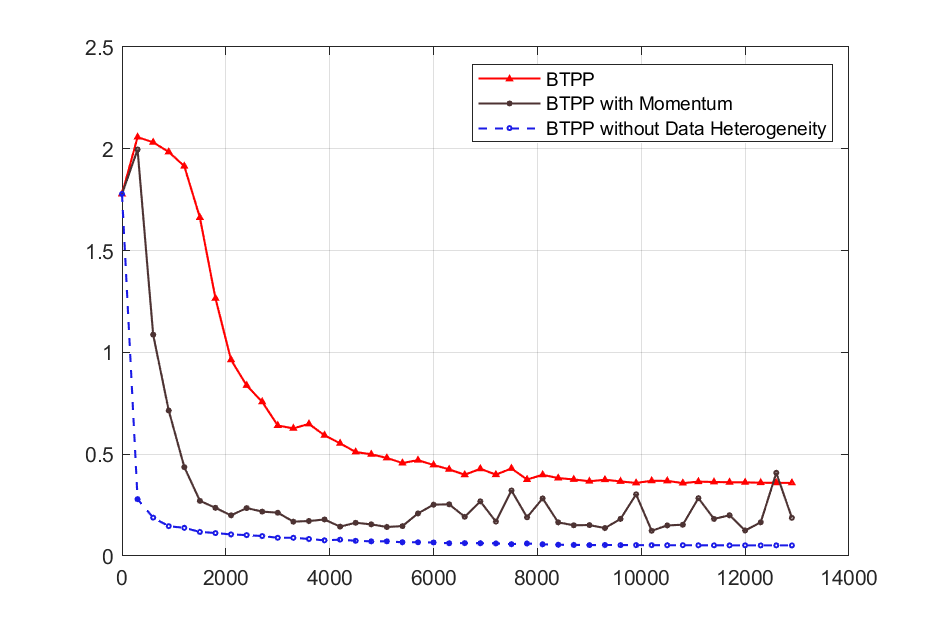} 
    \end{minipage}
    \hfill 
    \begin{minipage}{0.49\textwidth}
        \centering
        \includegraphics[width=\linewidth]{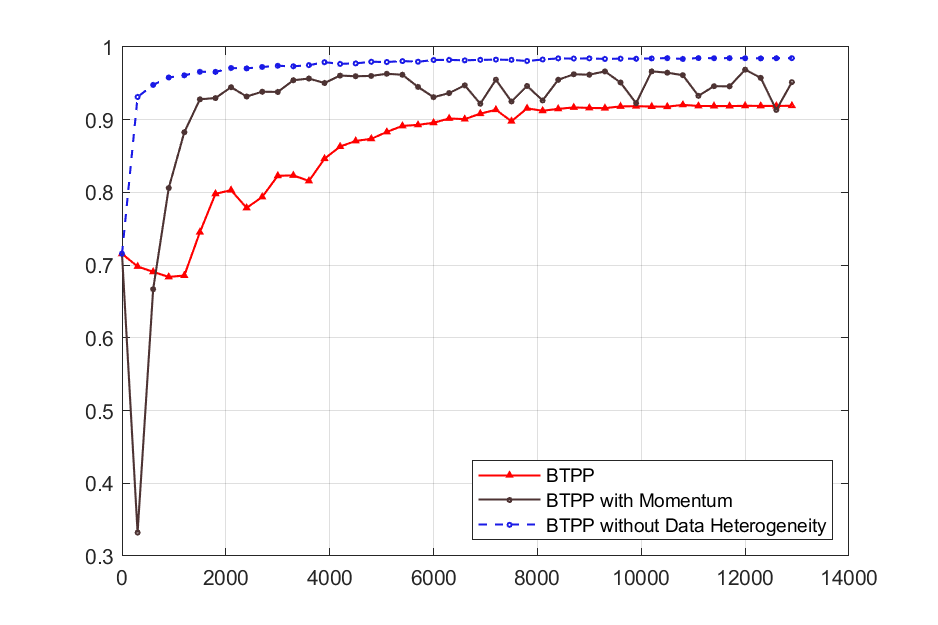} 
    \end{minipage}
    \caption{Performance of BTPP with branch size $B=2$ under various configurations when training CNN over MNIST.}
    \label{fig:cnn_compare}
\end{figure}
 \begin{figure}[!htbp]
    \centering
    \begin{minipage}{0.49\textwidth}
        \centering
        \includegraphics[width=\linewidth]{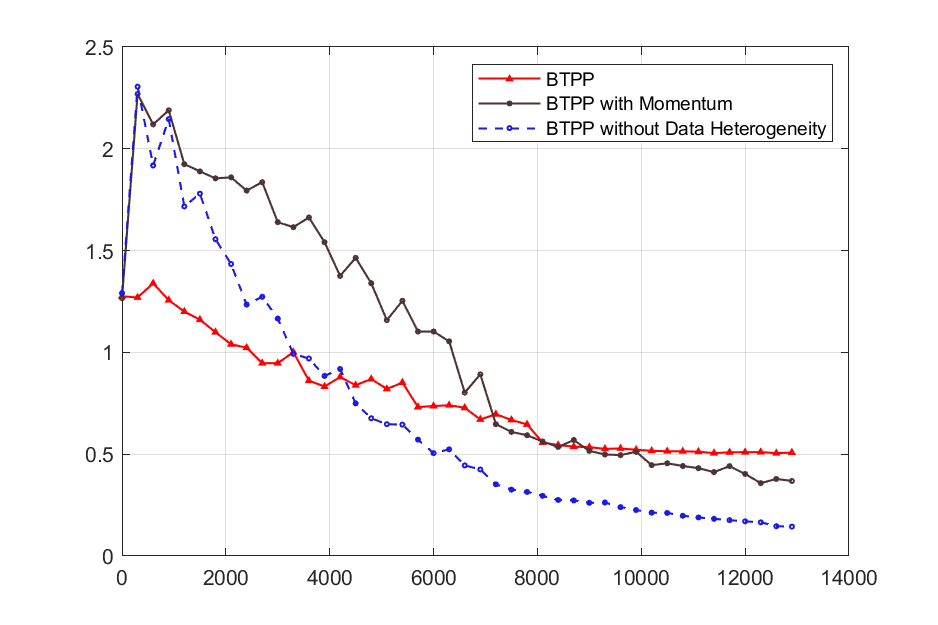} 
    \end{minipage}
    \hfill 
    \begin{minipage}{0.49\textwidth}
        \centering
        \includegraphics[width=\linewidth]{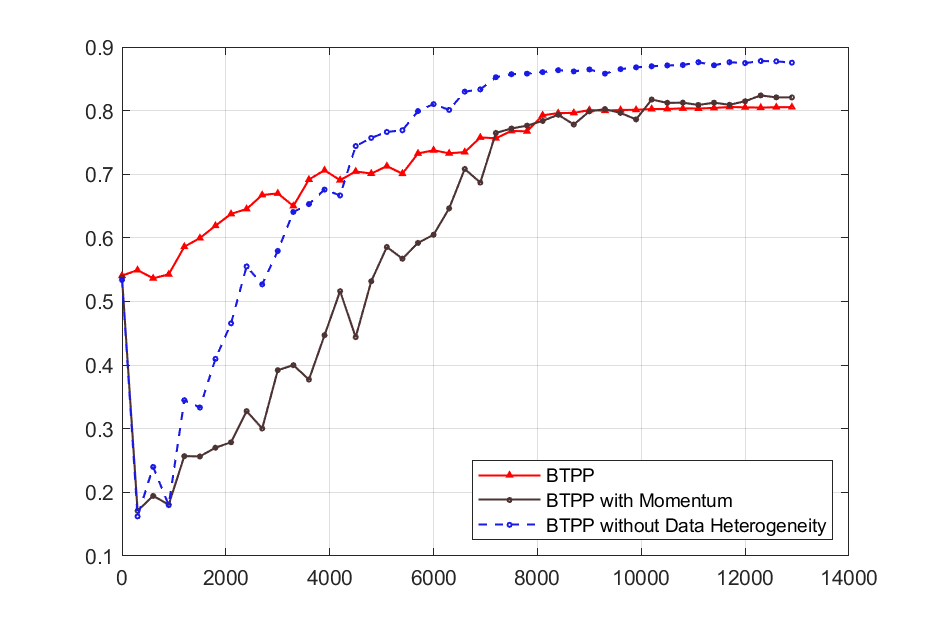} 
    \end{minipage}
    \caption{Performance of BTPP with branch size $B=2$ under various configurations when training VGG13 over CIFAR10.}
    \label{fig:vgg_compare}
\end{figure}


\newpage
\section*{NeurIPS Paper Checklist}

\begin{enumerate}

\item {\bf Claims}
    \item[] Question: Do the main claims made in the abstract and introduction accurately reflect the paper's contributions and scope?
    \item[] Answer:  \answerYes{} 
    \item[] Justification:  Our paper meticulously delineates its primary contributions in both the abstract (lines 8-9) and the introduction (subsection 1.2, lines 105-121).
    \item[] Guidelines:
    \begin{itemize}
        \item The answer NA means that the abstract and introduction do not include the claims made in the paper.
        \item The abstract and/or introduction should clearly state the claims made, including the contributions made in the paper and important assumptions and limitations. A No or NA answer to this question will not be perceived well by the reviewers. 
        \item The claims made should match theoretical and experimental results, and reflect how much the results can be expected to generalize to other settings. 
        \item It is fine to include aspirational goals as motivation as long as it is clear that these goals are not attained by the paper. 
    \end{itemize}

\item {\bf Limitations}
    \item[] Question: Does the paper discuss the limitations of the work performed by the authors?
    \item[] Answer: \answerYes{} 
    \item[] Justification:  At the end of section 1.1, lines 100-104.
    \item[] Guidelines:
    \begin{itemize}
        \item The answer NA means that the paper has no limitation while the answer No means that the paper has limitations, but those are not discussed in the paper. 
        \item The authors are encouraged to create a separate "Limitations" section in their paper.
        \item The paper should point out any strong assumptions and how robust the results are to violations of these assumptions (e.g., independence assumptions, noiseless settings, model well-specification, asymptotic approximations only holding locally). The authors should reflect on how these assumptions might be violated in practice and what the implications would be.
        \item The authors should reflect on the scope of the claims made, e.g., if the approach was only tested on a few datasets or with a few runs. In general, empirical results often depend on implicit assumptions, which should be articulated.
        \item The authors should reflect on the factors that influence the performance of the approach. For example, a facial recognition algorithm may perform poorly when image resolution is low or images are taken in low lighting. Or a speech-to-text system might not be used reliably to provide closed captions for online lectures because it fails to handle technical jargon.
        \item The authors should discuss the computational efficiency of the proposed algorithms and how they scale with dataset size.
        \item If applicable, the authors should discuss possible limitations of their approach to address problems of privacy and fairness.
        \item While the authors might fear that complete honesty about limitations might be used by reviewers as grounds for rejection, a worse outcome might be that reviewers discover limitations that aren't acknowledged in the paper. The authors should use their best judgment and recognize that individual actions in favor of transparency play an important role in developing norms that preserve the integrity of the community. Reviewers will be specifically instructed to not penalize honesty concerning limitations.
    \end{itemize}

\item {\bf Theory Assumptions and Proofs}
    \item[] Question: For each theoretical result, does the paper provide the full set of assumptions and a complete (and correct) proof?
    \item[] Answer: \answerYes{} 
    \item[] Justification:  The assumptions are in subsection 1.3, lines 129-132. The proofs of all theoretical results are in the Appendix and subsection 2.1.
    \item[] Guidelines:
    \begin{itemize}
        \item The answer NA means that the paper does not include theoretical results. 
        \item All the theorems, formulas, and proofs in the paper should be numbered and cross-referenced.
        \item All assumptions should be clearly stated or referenced in the statement of any theorems.
        \item The proofs can either appear in the main paper or the supplemental material, but if they appear in the supplemental material, the authors are encouraged to provide a short proof sketch to provide intuition. 
        \item Inversely, any informal proof provided in the core of the paper should be complemented by formal proofs provided in appendix or supplemental material.
        \item Theorems and Lemmas that the proof relies upon should be properly referenced. 
    \end{itemize}

    \item {\bf Experimental Result Reproducibility}
    \item[] Question: Does the paper fully disclose all the information needed to reproduce the main experimental results of the paper to the extent that it affects the main claims and/or conclusions of the paper (regardless of whether the code and data are provided or not)?
    \item[] Answer: \answerYes{} 
    \item[] Justification:  We present the experiment details in section 4 and upload our code in the link shown in the abstract.
    \item[] Guidelines:
    \begin{itemize}
        \item The answer NA means that the paper does not include experiments.
        \item If the paper includes experiments, a No answer to this question will not be perceived well by the reviewers: Making the paper reproducible is important, regardless of whether the code and data are provided or not.
        \item If the contribution is a dataset and/or model, the authors should describe the steps taken to make their results reproducible or verifiable. 
        \item Depending on the contribution, reproducibility can be accomplished in various ways. For example, if the contribution is a novel architecture, describing the architecture fully might suffice, or if the contribution is a specific model and empirical evaluation, it may be necessary to either make it possible for others to replicate the model with the same dataset, or provide access to the model. In general. releasing code and data is often one good way to accomplish this, but reproducibility can also be provided via detailed instructions for how to replicate the results, access to a hosted model (e.g., in the case of a large language model), releasing of a model checkpoint, or other means that are appropriate to the research performed.
        \item While NeurIPS does not require releasing code, the conference does require all submissions to provide some reasonable avenue for reproducibility, which may depend on the nature of the contribution. For example
        \begin{enumerate}
            \item If the contribution is primarily a new algorithm, the paper should make it clear how to reproduce that algorithm.
            \item If the contribution is primarily a new model architecture, the paper should describe the architecture clearly and fully.
            \item If the contribution is a new model (e.g., a large language model), then there should either be a way to access this model for reproducing the results or a way to reproduce the model (e.g., with an open-source dataset or instructions for how to construct the dataset).
            \item We recognize that reproducibility may be tricky in some cases, in which case authors are welcome to describe the particular way they provide for reproducibility. In the case of closed-source models, it may be that access to the model is limited in some way (e.g., to registered users), but it should be possible for other researchers to have some path to reproducing or verifying the results.
        \end{enumerate}
    \end{itemize}

\item {\bf Open access to data and code}
    \item[] Question: Does the paper provide open access to the data and code, with sufficient instructions to faithfully reproduce the main experimental results, as described in supplemental material?
    \item[] Answer: \answerYes{} 
    \item[] Justification: We provide our code with the link shown in lines 10-11.
    \item[] Guidelines:
    \begin{itemize}
        \item The answer NA means that paper does not include experiments requiring code.
        \item Please see the NeurIPS code and data submission guidelines (\url{https://nips.cc/public/guides/CodeSubmissionPolicy}) for more details.
        \item While we encourage the release of code and data, we understand that this might not be possible, so “No” is an acceptable answer. Papers cannot be rejected simply for not including code, unless this is central to the contribution (e.g., for a new open-source benchmark).
        \item The instructions should contain the exact command and environment needed to run to reproduce the results. See the NeurIPS code and data submission guidelines (\url{https://nips.cc/public/guides/CodeSubmissionPolicy}) for more details.
        \item The authors should provide instructions on data access and preparation, including how to access the raw data, preprocessed data, intermediate data, and generated data, etc.
        \item The authors should provide scripts to reproduce all experimental results for the new proposed method and baselines. If only a subset of experiments are reproducible, they should state which ones are omitted from the script and why.
        \item At submission time, to preserve anonymity, the authors should release anonymized versions (if applicable).
        \item Providing as much information as possible in supplemental material (appended to the paper) is recommended, but including URLs to data and code is permitted.
    \end{itemize}

\item {\bf Experimental Setting/Details}
    \item[] Question: Does the paper specify all the training and test details (e.g., data splits, hyperparameters, how they were chosen, type of optimizer, etc.) necessary to understand the results?
    \item[] Answer: \answerYes{} 
    \item[] Justification: We provide all the experimental details and settings in Section 4 and our code link.
    \item[] Guidelines:
    \begin{itemize}
        \item The answer NA means that the paper does not include experiments.
        \item The experimental setting should be presented in the core of the paper to a level of detail that is necessary to appreciate the results and make sense of them.
        \item The full details can be provided either with the code, in appendix, or as supplemental material.
    \end{itemize}

\item {\bf Experiment Statistical Significance}
    \item[] Question: Does the paper report error bars suitably and correctly defined or other appropriate information about the statistical significance of the experiments?
    \item[] Answer: \answerNo{} 
    \item[] Justification: We have repeated some experiments with different seeds
and reported the averaged performance. See Section 4 for details.
    \item[] Guidelines:
    \begin{itemize}
        \item The answer NA means that the paper does not include experiments.
        \item The authors should answer "Yes" if the results are accompanied by error bars, confidence intervals, or statistical significance tests, at least for the experiments that support the main claims of the paper.
        \item The factors of variability that the error bars are capturing should be clearly stated (for example, train/test split, initialization, random drawing of some parameter, or overall run with given experimental conditions).
        \item The method for calculating the error bars should be explained (closed form formula, call to a library function, bootstrap, etc.)
        \item The assumptions made should be given (e.g., Normally distributed errors).
        \item It should be clear whether the error bar is the standard deviation or the standard error of the mean.
        \item It is OK to report 1-sigma error bars, but one should state it. The authors should preferably report a 2-sigma error bar than state that they have a 96\% CI, if the hypothesis of Normality of errors is not verified.
        \item For asymmetric distributions, the authors should be careful not to show in tables or figures symmetric error bars that would yield results that are out of range (e.g. negative error rates).
        \item If error bars are reported in tables or plots, The authors should explain in the text how they were calculated and reference the corresponding figures or tables in the text.
    \end{itemize}

\item {\bf Experiments Compute Resources}
    \item[] Question: For each experiment, does the paper provide sufficient information on the computer resources (type of compute workers, memory, time of execution) needed to reproduce the experiments?
    \item[] Answer: \answerYes{} 
    \item[] Justification: See line 255 for reference.
    \item[] Guidelines:
    \begin{itemize}
        \item The answer NA means that the paper does not include experiments.
        \item The paper should indicate the type of compute workers CPU or GPU, internal cluster, or cloud provider, including relevant memory and storage.
        \item The paper should provide the amount of compute required for each of the individual experimental runs as well as estimate the total compute. 
        \item The paper should disclose whether the full research project required more compute than the experiments reported in the paper (e.g., preliminary or failed experiments that didn't make it into the paper). 
    \end{itemize}
    
\item {\bf Code Of Ethics}
    \item[] Question: Does the research conducted in the paper conform, in every respect, with the NeurIPS Code of Ethics \url{https://neurips.cc/public/EthicsGuidelines}?
    \item[] Answer: \answerYes{} 
    \item[] Justification: The research conducted in the paper conforms, in every respect, with the NeurIPS Code of Ethics.
    \item[] Guidelines:
    \begin{itemize}
        \item The answer NA means that the authors have not reviewed the NeurIPS Code of Ethics.
        \item If the authors answer No, they should explain the special circumstances that require a deviation from the Code of Ethics.
        \item The authors should make sure to preserve anonymity (e.g., if there is a special consideration due to laws or regulations in their jurisdiction).
    \end{itemize}

\item {\bf Broader Impacts}
    \item[] Question: Does the paper discuss both potential positive societal impacts and negative societal impacts of the work performed?
    \item[] Answer: \answerNA{} 
    \item[] Justification: There is no societal impact of the work performed.
    \item[] Guidelines:
    \begin{itemize}
        \item The answer NA means that there is no societal impact of the work performed.
        \item If the authors answer NA or No, they should explain why their work has no societal impact or why the paper does not address societal impact.
        \item Examples of negative societal impacts include potential malicious or unintended uses (e.g., disinformation, generating fake profiles, surveillance), fairness considerations (e.g., deployment of technologies that could make decisions that unfairly impact specific groups), privacy considerations, and security considerations.
        \item The conference expects that many papers will be foundational research and not tied to particular applications, let alone deployments. However, if there is a direct path to any negative applications, the authors should point it out. For example, it is legitimate to point out that an improvement in the quality of generative models could be used to generate deepfakes for disinformation. On the other hand, it is not needed to point out that a generic algorithm for optimizing neural networks could enable people to train models that generate Deepfakes faster.
        \item The authors should consider possible harms that could arise when the technology is being used as intended and functioning correctly, harms that could arise when the technology is being used as intended but gives incorrect results, and harms following from (intentional or unintentional) misuse of the technology.
        \item If there are negative societal impacts, the authors could also discuss possible mitigation strategies (e.g., gated release of models, providing defenses in addition to attacks, mechanisms for monitoring misuse, mechanisms to monitor how a system learns from feedback over time, improving the efficiency and accessibility of ML).
    \end{itemize}
    
\item {\bf Safeguards}
    \item[] Question: Does the paper describe safeguards that have been put in place for responsible release of data or models that have a high risk for misuse (e.g., pretrained language models, image generators, or scraped datasets)?
    \item[] Answer: \answerNA{} 
    \item[] Justification: The paper poses no such risks.
    \item[] Guidelines:
    \begin{itemize}
        \item The answer NA means that the paper poses no such risks.
        \item Released models that have a high risk for misuse or dual-use should be released with necessary safeguards to allow for controlled use of the model, for example by requiring that users adhere to usage guidelines or restrictions to access the model or implementing safety filters. 
        \item Datasets that have been scraped from the Internet could pose safety risks. The authors should describe how they avoided releasing unsafe images.
        \item We recognize that providing effective safeguards is challenging, and many papers do not require this, but we encourage authors to take this into account and make a best faith effort.
    \end{itemize}

\item {\bf Licenses for existing assets}
    \item[] Question: Are the creators or original owners of assets (e.g., code, data, models), used in the paper, properly credited and are the license and terms of use explicitly mentioned and properly respected?
    \item[] Answer: \answerYes{} 
    \item[] Justification: We have cited the original paper that produced the code package or dataset.
    \item[] Guidelines:
    \begin{itemize}
        \item The answer NA means that the paper does not use existing assets.
        \item The authors should cite the original paper that produced the code package or dataset.
        \item The authors should state which version of the asset is used and, if possible, include a URL.
        \item The name of the license (e.g., CC-BY 4.0) should be included for each asset.
        \item For scraped data from a particular source (e.g., website), the copyright and terms of service of that source should be provided.
        \item If assets are released, the license, copyright information, and terms of use in the package should be provided. For popular datasets, \url{paperswithcode.com/datasets} has curated licenses for some datasets. Their licensing guide can help determine the license of a dataset.
        \item For existing datasets that are re-packaged, both the original license and the license of the derived asset (if it has changed) should be provided.
        \item If this information is not available online, the authors are encouraged to reach out to the asset's creators.
    \end{itemize}

\item {\bf New Assets}
    \item[] Question: Are new assets introduced in the paper well documented and is the documentation provided alongside the assets?
    \item[] Answer:\answerNA{} 
    \item[] Justification: The paper does not release new assets.
    \item[] Guidelines:
    \begin{itemize}
        \item The answer NA means that the paper does not release new assets.
        \item Researchers should communicate the details of the dataset/code/model as part of their submissions via structured templates. This includes details about training, license, limitations, etc. 
        \item The paper should discuss whether and how consent was obtained from people whose asset is used.
        \item At submission time, remember to anonymize your assets (if applicable). You can either create an anonymized URL or include an anonymized zip file.
    \end{itemize}

\item {\bf Crowdsourcing and Research with Human Subjects}
    \item[] Question: For crowdsourcing experiments and research with human subjects, does the paper include the full text of instructions given to participants and screenshots, if applicable, as well as details about compensation (if any)? 
    \item[] Answer: \answerNA{} 
    \item[] Justification: The paper does not involve crowdsourcing nor research with human subjects.
    \item[] Guidelines:
    \begin{itemize}
        \item The answer NA means that the paper does not involve crowdsourcing nor research with human subjects.
        \item Including this information in the supplemental material is fine, but if the main contribution of the paper involves human subjects, then as much detail as possible should be included in the main paper. 
        \item According to the NeurIPS Code of Ethics, workers involved in data collection, curation, or other labor should be paid at least the minimum wage in the country of the data collector. 
    \end{itemize}

\item {\bf Institutional Review Board (IRB) Approvals or Equivalent for Research with Human Subjects}
    \item[] Question: Does the paper describe potential risks incurred by study participants, whether such risks were disclosed to the subjects, and whether Institutional Review Board (IRB) approvals (or an equivalent approval/review based on the requirements of your country or institution) were obtained?
    \item[] Answer: \answerNA{} 
    \item[] Justification: The paper does not involve crowdsourcing nor research with human subjects.
    \item[] Guidelines:
    \begin{itemize}
        \item The answer NA means that the paper does not involve crowdsourcing nor research with human subjects.
        \item Depending on the country in which research is conducted, IRB approval (or equivalent) may be required for any human subjects research. If you obtained IRB approval, you should clearly state this in the paper. 
        \item We recognize that the procedures for this may vary significantly between institutions and locations, and we expect authors to adhere to the NeurIPS Code of Ethics and the guidelines for their institution. 
        \item For initial submissions, do not include any information that would break anonymity (if applicable), such as the institution conducting the review.
    \end{itemize}

\end{enumerate}

\end{document}